\documentclass[12pt]{amsart}
\usepackage[utf8]{inputenc}
\usepackage[margin=3cm]{geometry}
\usepackage{bbm,amssymb,mathtools,mathrsfs}
\usepackage[dvipsnames]{xcolor}
\usepackage{graphicx,titletoc}
\usepackage[shortlabels]{enumitem}
\definecolor{darkblue}{rgb}{0.0, 0.0, 0.55}
\definecolor{bordeaux}{rgb}{0.34, 0.01, 0.1}
\usepackage[colorlinks,linkcolor=bordeaux,citecolor=darkblue,urlcolor=black,hypertexnames=true]{hyperref}

\newtheorem{thm}{Theorem}[section]
\newtheorem{cor}[thm]{Corollary}
\newtheorem{lem}[thm]{Lemma}
\newtheorem{prop}[thm]{Proposition}
\theoremstyle{definition}

\newtheorem{rem}[thm]{Remark}
\newtheorem{exa}[thm]{Example}

\DeclareMathOperator{\rk}{rk}
\DeclareMathOperator{\ran}{ran}

\DeclareMathOperator{\opm}{M}
\DeclareMathOperator{\oph}{H}
\DeclareMathOperator{\GL}{GL}
\DeclareMathOperator{\End}{End}
\DeclareMathOperator{\Hom}{Hom}
\DeclareMathOperator{\Ext}{Ext}

\def\C{\mathbb C}
\def\R{\mathbb R}
\def\N{\mathbb N}

\def\cA{\mathcal A}
\def\cB{\mathcal B}

\def\cC{\mathcal C}
\def\cD{\mathcal D}

\def\cH{\mathcal H}

\def\cL{\mathcal L}
\def\cM{\mathcal M}

\def\cO{\mathcal O}

\def\cZ{\mathcal Z}

\def\kk{{\mathbbm k}}

\def\ux{{\underline x}}
\def\uy{{\underline y}}
\def\uz{{\underline z}}
\def\uxs{{\underline x^*}}
\def\uA{{\underline A}}
\def\uB{{\underline B}}
\def\uX{{\underline X}}
\def\uY{{\underline Y}}
\def\uZ{{\underline Z}}

\newcommand*{\mtx}[1]{\opm_{#1}(\kk)}
\newcommand*{\mtxc}[1]{\opm_{#1}(\C)}
\newcommand*{\her}[1]{\oph_{#1}(\C)}
\newcommand{\Langle}{\mathop{<}\!}
\newcommand{\Rangle}{\!\mathop{>}}

\newcommand{\px}{\kk\!\Langle \ux\Rangle}
\newcommand{\cx}{\C\!\Langle \ux\Rangle}
\newcommand{\py}{\kk\!\Langle \uy\Rangle}
\newcommand{\cxs}{\C\!\Langle \ux, \uxs \Rangle}

\usepackage{letltxmacro}
\LetLtxMacro\orgvdots\vdots
\LetLtxMacro\orgddots\ddots

\makeatletter
\DeclareRobustCommand\vdots{%
  \mathpalette\@vdots{}%
}
\newcommand*{\@vdots}[2]{%
  \sbox0{$#1\cdotp\cdotp\cdotp\m@th$}%
  \sbox2{$#1.\m@th$}%
  \vbox{%
    \dimen@=\wd0 %
    \advance\dimen@ -3\ht2 %
    \kern.5\dimen@
    \dimen@=\wd2 %
    \advance\dimen@ -\ht2 %
    \dimen2=\wd0 %
    \advance\dimen2 -\dimen@
    \vbox to \dimen2{%
      \offinterlineskip
      \copy2 \vfill\copy2 \vfill\copy2 %
    }%
  }%
}
\DeclareRobustCommand\ddots{%
  \mathinner{%
    \mathpalette\@ddots{}%
    \mkern\thinmuskip
  }%
}
\newcommand*{\@ddots}[2]{%
  \sbox0{$#1\cdotp\cdotp\cdotp\m@th$}%
  \sbox2{$#1.\m@th$}%
  \vbox{%
    \dimen@=\wd0 %
    \advance\dimen@ -3\ht2 %
    \kern.5\dimen@
    \dimen@=\wd2 %
    \advance\dimen@ -\ht2 %
    \dimen2=\wd0 %
    \advance\dimen2 -\dimen@
    \vbox to \dimen2{%
      \offinterlineskip
      \hbox{$#1\mathpunct{.}\m@th$}%
      \vfill
      \hbox{$#1\mathpunct{\kern\wd2}\mathpunct{.}\m@th$}%
      \vfill
      \hbox{$#1\mathpunct{\kern\wd2}\mathpunct{\kern\wd2}\mathpunct{.}\m@th$}%
    }%
  }%
}
\makeatother

\title[Pointwise equivalences of nc polynomials]{Global structure behind pointwise equivalences of noncommutative polynomials}

\author{Eli Shamovich}
\address{Department of Mathematics, Ben-Gurion University of the Negev, Israel}
\email{shamovic@bgu.ac.il}
\thanks{Supported by the BSF Grant no. 2022235.}

\author{Jurij Vol\v{c}i\v{c}}
\address{Department of Mathematics, University of Auckland, New Zealand}
\email{jurij.volcic@auckland.ac.nz}
\thanks{Supported by the NSF grant DMS-1954709, the BSF Grant no. 2022235, the Marsden Fund MFP-UOA2528 from Government funding administered by the Royal Society Te Ap\=arangi, and the New Staff grant 3734441 by the Faculty of Science, the University of Auckland.}
\date{\today}



\begin{document}

\begin{abstract}
This paper investigates the interplay between local and global equivalences on noncommutative polynomials, the elements of the free algebra. When the latter are viewed as functions in several matrix variables, a local equivalence of noncommutative polynomials refers to their values sharing a common feature point-wise on matrix tuples of all dimensions, such as rank-equivalence (values have the same ranks), isospectrality (values have the same spectrum), and pointwise similarity (values are similar). On the other hand, a global equivalence refers to a ring-theoretic relation within the free algebra, such as stable association or (elementary) intertwinedness. This paper identifies the most ubiquitous pairs of local and global equivalences. Namely, rank-equivalence coincides with stable association, isospectrality coincides with both intertwinedness and transitive closure of elementary intertwinedness, and pointwise similarity coincides with equality. Using these characterizations, further results on spectral radii and norms of values of noncommutative polynomials are derived.
\end{abstract}

\maketitle

\tableofcontents

\section{Introduction}

The free algebra $\px = \kk\!\Langle x_1,\ldots,x_n \Rangle $ over a field $\kk$ is, in many ways, a good noncommutative generalization of the polynomial ring. The algebraic properties of the free algebra and, more generally, free ideal rings were extensively studied by Cohn \cite{cohn}, Amitsur (see, for example, \cite{Amitsur-null, Amitsur-rat_iden}), and others (see, for example, \cite{ber69, FR, Schofield-book}). The main goal of this paper is to study the algebraic properties through a more geometric and analytic approach that originates in noncommutative analysis. Broadly speaking, noncommutative analysis is an area of analysis dealing with operator algebras and related concepts. Recently, a fruitful approach to operator algebras (both self-adjoint and not) was to view them as functions on the space of their representations. This idea is, of course, not new and goes back to Gelfand. In the theory of $C^*$-algebras, noncommutative functions first appeared in the work of Takesaki \cite{Takesaki} that was extended by Bichteler \cite{Bichteler}. More recently, this point of view was enhanced and applied to the study of operator systems and noncommutative convexity by Davidson and Kennedy \cite{DavidsonKennedy_choquet}. For our purposes here, it suffices to consider noncommutative functions on the noncommutative affine space. Let 
$$\cM^n=\bigcup_{k\in\N} \mtx{k}^n$$
be the set of all $n$-tuples of square matrices over $\kk$. This space, clearly, parametrizes all finite-dimensional representations of $\px$. An element $f \in \px$ gives rise to a function on $\cM^n$ that takes values in $\cM^1$. Moreover, $f$ satisfies the following three properties:
\begin{itemize}
    \item For all $k \in \N$, $f(\mtx{k}^n) \subseteq \mtx{k}$;
    \item For all $\uX \in \mtx{k}^n$ and $\uY \in \mtx{\ell}^n$, define 
    \[
    \uX \oplus \uY = \left(\begin{pmatrix} X_1 & 0 \\ 0 & Y_1 \end{pmatrix}, \ldots, \begin{pmatrix} X_n & 0 \\ 0 & Y_n \end{pmatrix} \right) \in \mtx{k+\ell}^n.
    \]
    Then, $f(\uX \oplus \uY) = f(\uX) \oplus f(\uY)$.
    \item For all $\uX \in \mtx{k}^n$ and all $S \in \GL_k(\kk)$, define
    \[
    S^{-1} \uX S = \left( S^{-1} X_1 S, \ldots, S^{-1} X_n S \right) \in \mtx{k}^n.
    \]
    Then, $f(S^{-1} \uX S) = S^{-1} f(\uX) S$.
\end{itemize}
These properties in the case of elements of $\px$ simply encode the behavior of representations. However, these properties (for $\kk = \C$) were introduced by Taylor to create algebras of functions that may give rise to functional calculus for general tuples of operators \cite{Taylor-frame, Taylor-ncfunc} (see also \cite{Luminet-PI, Luminet-func}).  Taylor's work was greatly developed in recent years by many researchers (see the books \cite{AMY-book,BallBolotnikov-book,KVV10}). Moreover, Popescu has developed a notion of a noncommutative function that arose naturally from his study of dilation theory and operator algebras generated by row isometries (see, for example, \cite{Popescu-funcI,Popescu-funcII}). Popescu's theory fits nicely into the framework of noncommutative analysis (see \cite{JuryMartin-fatou,JuryMartin-lebesgue,SSS1,SampatShalit25} for examples of applications). In all this theory, the free algebra plays a role similar to that of the polynomial ring in the theory of analytic functions. Therefore, we will call elements of $\px$ \textit{noncommutative polynomials} or simply polynomials, if there can be no confusion.

Another natural area of applications of noncommutative functions is free probability. In fact, Voiculescu has independently discovered noncommutative functions \cite{Voiculescu-subordination,Voiculescu-quest1,Voiculescu-quest2}. One of the main reasons was the study of operator-valued free probability. One of Voiculescu's key results is a free version of a central limit theorem \cite{Voiculescu-quasi_diag}. Roughly speaking, the result says that the limit of the expected value of the trace of a noncommutative polynomial in independent GUE matrices converges to the trace of the polynomial applied to free semicirculars as the size of the matrices tends to infinity. This result was extended by Haagerup and Thorbj{\o}rnsen \cite{HaagThor-ext} to almost everywhere convergence of norms. Later, this result was extended to noncommutative rational functions (elements of the free skew field) by Mai, Speicher, and Yin \cite{MaiSpeYin}, and the ranks of their evaluations were used to investigate unavoidable atoms in noncommutative distributions in \cite{arizmendi}.

From a more geometric point of view, one can consider various types of varieties in $\cM^n$ and its subdomains. In \cite{HKV} and \cite{HKV0}, Helton, Klep, and the second author have studied determinantal varieties. Namely, given $f \in \px$, one can define $\cZ(f) = \{\uX \in \cM^n \mid \det f(\uX) = 0\}$. The authors proved that for $f, g \in \px$, if $\cZ(f) = \cZ(g)$, then they admit factorization $f = f_1 \cdots f_k$ and $g = g_1 \cdots g_k$ in $\px$, such that for each $1 \leq i \leq k$, there exists $1 \leq j \leq k$, such that $f_i$ and $g_j$ are stably associated. Cohn introduced stable association to study the problem of factorizations in $\px$. We say that two polynomials $f, g \in \px$ are \emph{stably associated}, if there exist $P, Q \in \GL_2(\px)$, such that
\[
\begin{pmatrix} g & 0 \\ 0 & 1 \end{pmatrix} = P \begin{pmatrix} f & 0 \\ 0 & 1 \end{pmatrix} Q.
\]
Another type of noncommutative varieties is the ``hard zeroes'' varieties. Namely, for $f \in \px$, we consider the collection of points where $f$ vanishes in the strong sense. In \cite{SSS1}, such varieties were studied in the unit ball as geometric invariants to classify certain operator algebras arising from noncommutative interpolation. Varieties of hard zeroes naturally correspond to two-sided ideals. A homogeneous Nullstellensatz was obtained in \cite{SSS1}. Lastly, one can consider directional varieties, which naturally correspond to left (or right) ideals. A natural Nullstellensatz is the so-called Bergman Nullstellensatz \cite{HelMcC-positivss}. Similar varieties appear in the analytic context to study factorizations in the free semigroup algebra \cite{JMS-bso}.

The goal of this paper is to study certain relations between polynomials described in terms of their values on $\cM^n$. Throughout the paper, we assume that the field $\kk$ is algebraically closed of characteristic 0. We say that noncommutative polynomials $f,g$ are:
\begin{enumerate}[(i)]
    \item \emph{rank-equivalent} if $\rk f(\uX)=\rk g(\uX)$ for all $\uX\in\cM^n$;
    \item \emph{isospectral} if the eigenvalues of $f(\uX)$ and $g(\uX)$ coincide for all $\uX\in\cM^n$;
    \item \emph{pointwise similar} if the matrices $f(\uX)$ and $g(\uX)$ are similar for all $\uX\in\cM^n$;
    \item ($\kk=\C$) \emph{pointwise the same norm $\|\cdot\|$} if $\|f(\uX)\|=\|g(\uX)\|$ for all $\uX\in\cM^n$.
\end{enumerate}
Note that the first condition can be thought of as the condition on classical analytic functions to have the same zeroes of the same order. The first condition is also related to the more general study of rank-stability of noncommutative polynomial equations \cite{beeri,ElekGrab}.
The second condition is, in a sense, generic similarity. The second condition first appeared in \cite{vol20}, where a noncommutative version of Bertini's theorem from algebraic geometry was obtained. The third condition is, of course, stronger than the second one. In fact, such questions arise from time to time in noncommutative analysis. Namely, given two elements of an operator algebra and a collection of representations that separate points, can we deduce that the two elements are similar or unitarily equivalent in the algebra, if their images under each one of our representations are similar or unitarily equivalent, respectively? Lastly, the last condition is metric. It is analogous to two analytic functions having the same modulus. It is well known that in this case, the two analytic functions differ by a multiplication by a unimodular scalar. Surprisingly, this is also true for noncommutative polynomials, as we will see.

As we have seen, Taylor introduced noncommutative functions to study a notion of joint spectrum and functional calculus for arbitrary tuples of operators on a Banach/Hilbert space. Taylor's definition of the joint spectrum has some drawbacks, however. Firstly, it deals with categories of Fr\'echet modules over certain topological algebras, and Taylor builds a sophisticated cohomological machinery to deal with these categories \cite{Taylor-homology}. Secondly, both in the case of commuting operators and the case of an arbitrary tuple of matrices, the spectrum and the functional calculus do not work as desired (see also \cite{Luminet-PI}). A more naive approach is to study the singularity loci of pencils in order to understand the ``joint spectrum''. Namely, one can think of the spectrum of an operator $T$ as the singularities in the complex plane of the analytic operator-valued function $z \mapsto (zI-T)^{-1}$. Homogenizing we get the function $(z I - w T)^{-1}$. More generally, the projective spectrum of a tuple of operators was considered by Yang and his collaborators \cite{CadeYang, Yang-proj_spec}. From the noncommutative analysis perspective, a natural notion of the joint spectrum of a tuple of matrices $A_1,\ldots,A_n \in \mtxc{k}$ is the singularity set of the pencil $I - \sum_{j=1}^n x_j A_j$ or its homogenization $x_0 I - \sum_{j=1}^n x_j A_j$. The latter relates to the noncommutative projective space that appeared first in the work of Voiculescu \cite{Voiculescu-quest2} and was recently used to solve a joint similarity problem for tuples of matrices by Derksen, Klep, Makam, and the second author \cite{DKMV}.  The former relates to the domains of definition of noncommutative rational functions and was studied in \cite{KlepVolcic-loci_rat}. More recently, realizations with operator coefficients were introduced by Augat, Martin, and the first author in \cite{AugMarS}. Therefore, understanding the local behaviour of noncommutative polynomials and pencils on $\cM^n$ is a natural first step in the study of joint spectra of arbitrary tuples of operators.

The following example lists some non-trivial instances of our relations.

\begin{exa}
Polynomials $f_1=xy+1$ and $g_1=yx+1$ are both rank-equivalent and isospectral, but not pointwise similar:
$$f_1\left(
\begin{pmatrix}
1&0\\0&0
\end{pmatrix},
\begin{pmatrix}
0&1\\0&0
\end{pmatrix}
\right)=\begin{pmatrix}
1&1\\0&1
\end{pmatrix},\quad
g_1\left(
\begin{pmatrix}
1&0\\0&0
\end{pmatrix},
\begin{pmatrix}
0&1\\0&0
\end{pmatrix}
\right)=\begin{pmatrix}
1&0\\0&1
\end{pmatrix}.$$
Polynomials $f_2=xy$ and $g_2=yx$ are isospectral, but not rank-equivalent:
$$f_2
\left(\begin{pmatrix}
1&0\\0&0
\end{pmatrix},
\begin{pmatrix}
0&1\\0&0
\end{pmatrix}\right)
=\begin{pmatrix}
0&1\\0&0
\end{pmatrix},\quad
g_2\left(\begin{pmatrix}
1&0\\0&0
\end{pmatrix},\begin{pmatrix}
0&1\\0&0
\end{pmatrix}\right)
=\begin{pmatrix}
0&0\\0&0
\end{pmatrix}.
$$
Polynomials $f_3=xyxy+xy+x$ and $g_3=xy^2x+xy+x$ are rank-equivalent, but not isospectral:
$$f_3
\left(\begin{pmatrix}
1&0\\0&0
\end{pmatrix},\begin{pmatrix}
0&1\\1&0
\end{pmatrix}\right)
=\begin{pmatrix}
1&1\\0&0
\end{pmatrix},\quad
g_3
\left(\begin{pmatrix}
1&0\\0&0
\end{pmatrix},\begin{pmatrix}
0&1\\1&0
\end{pmatrix}\right)
=\begin{pmatrix}
2&1\\0&0
\end{pmatrix}.
$$

The claimed relations may be verified using elementary means; from the perspective of this paper, they are consequences of relations $f_1x=xg_1$, $f_2x=xg_2$, $f_3(yx+1)=(xy+1)g_3$ 
and Theorems \ref{t:master}, \ref{t:st_assoc} below.
\end{exa}

We say that two polynomials $f,g \in \px$ are \emph{intertwined} if there exists a nonzero $a \in \px$ such that $f a = a g$. This relation first appeared in \cite{vol20}. The main result of this paper is the following theorem. Note that the reverse implications in the statement are trivial.

\begin{thm}\label{t:master}
Let $f,g$ be noncommutative polynomials.
\begin{enumerate}[(i)]
\item $f$ and $g$ are rank-equivalent if and only if they are stably associated (Theorem \ref{t:rankequiv}).
\item $f$ and $g$ are isospectral if and only if they are intertwined (Theorem \ref{t:isospec}).
\item $f$ and $g$ are pointwise similar if and only if they are equal (Theorem \ref{t:similar}).
\item $f$ and $g$ have pointwise the same operator norm if and only if they are equal up to scaling by a constant of modulus 1 (Theorem \ref{t:realmix}).
\end{enumerate}
\end{thm}

While Theorem \ref{t:master} sums up the main contributions of the paper, we also provide additional insight into these local equivalences through examples, counterexamples, and their finer features.
Two problems are left open. We say that $f$ and $g$ are \emph{operator isospectral} if $f(\uX)$ has the same spectrum as $g(\uX)$ for every tuple $\uX \in B(\cH)^n$, where $\cH$ is a separable Hilbert space. In Examples \ref{ex:not_op_isosp} and \ref{ex:opspec}, we show that this relation is neither isospectrality nor equality. It would be interesting to understand this relation algebraically. Another intriguing question is to extend some of these results to certain operator algebras. For example, it stands to reason that Theorem \ref{t:master}(iv) may admit a generalization to algebras of noncommutative functions on operator space balls or more general matrix convex sets. Furthermore, the ring of germs of noncommutative functions was shown to be a semifir (semi-free ideal ring) by Klep, Vinnikov, and the second author \cite{KlVinVol-germs}, meaning that it admits a well-behaved factorization theory, which is one of the main tools in this paper.
It would be interesting to see whether the results on isospectrality and stable association can be extended to this ring.

\section{Preliminaries on factorization of noncommutative polynomials}\label{sec:prelim}

This paper relies heavily on the profound and comprehensive theory of factorization in free algebras (and related rings), of which P. M. Cohn was the chief architect \cite{cohn}. In this section, we review some of the concepts and results from the factorization theory that are regularly applied throughout the paper.

Let $\kk$ be an algebraically closed field of characteristic 0,\footnote{
We are mainly interested in $\kk=\C$; some of the statements likely remain valid for more general fields.
} 
and let $x_1,\dots,x_n$ be freely noncommuting variables. Let $\px$ be the free associative $\kk$-algebra generated by $x_1,\dots,x_n$. Throughout the paper (especially within the proofs), we regularly denote $R=\px$. Its elements are called \emph{noncommutative polynomials}. We typically write $x_1=x$ and $x_2=y$ in examples with two variables. 
This paper revolves around matrix (and sometimes operator) evaluations of noncommutative polynomials. If $\uX=(X_1,\dots,X_n)\in\mtx{k}^n$, then $f(\uX)\in\mtx{k}$ is obtained in a natural way, by replacing the $x_j$ in $f$ by $X_j$ (and the constant term with the corresponding scalar multiple of the identity matrix).

Two polynomials $f,g\in\px $ are \emph{stably associated} \cite[Section 0.5]{cohn} if there exist $P,Q\in\GL_2(\px)$ such that $g\oplus 1=P(f\oplus 1)Q$. 
This notion makes up for the fact that the usual association is very restrictive because there are no invertible elements in $\px$ besides constants; as there is a plethora of invertible $2\times 2$ matrices over the free algebra, stable association is more flexible.
More generally, $F\in\px^{d\times d}$ and $G\in\px^{e\times e}$ are stably associated if there exist $P,Q\in\GL_{d+e}(\px)$ such that $G\oplus I_d=P(F\oplus I_e)Q$.
Observe that stably associated polynomials are clearly rank-equivalent. 
Stable association is indispensable for investigating factorization in $\px$. For example, while a complete factorization of a noncommutative polynomial is not unique in the strict sense (e.g., $x(yx+1)=(xy+1)x$), it is known that a complete factorization is unique up to permutation and stable association of the irreducible factors \cite[Proposition 3.2.9]{cohn}.

Next, $f,g\in\px$ are
\emph{right/left coprime} if they do not have a nonconstant common right/left factor in $\px$,
and they are \emph{right/left comaximal} if the right/left ideal in $\px$ generated by them equals $\px$.
A relation $fa=bg$ in $\px$ is \emph{coprime} if $f,b$ are left coprime and $a,g$ are right coprime, 
and \emph{comaximal} if $f,b$ are right comaximal and $a,g$ are left comaximal.
By \cite[Corollary 3.1.4]{cohn}, a relation in $\px$ is coprime if and only if it is comaximal, which we regularly use in this paper.

\begin{thm}{\cite[Corollary 0.5.5, Proposition 0.5.6 and Corollary 3.1.4]{cohn}}
\label{t:st_assoc}
The following are equivalent for $f,g\in\px$: 
\begin{enumerate}[(i)]
\item $f$ and $g$ are stably associated;
\item the left $\px$-modules $\px/\px\cdot f$ and $\px/\px\cdot g$ are isomorphic;
\item there is a coprime/comaximal relation $fa=bg$ in $\px$.
\end{enumerate}
\end{thm}

Throughout the paper, we interchangeably apply the interpretations of stable association from Theorem \ref{t:st_assoc}; while the primary definition above is most directly related to rank equivalence, the others are more convenient in proofs. While not essential to this paper, let us mention that it is possible to generate all pairs of stably associated polynomials using \emph{continuant polynomials} \cite[Section 2.7]{cohn}.

Let $a\in\px$. The \emph{(left) eigenring} \cite[Section 0.6]{cohn} of $a$ is
\begin{align*}
E(a)\,:&=\End_{\px}\big(\px/\px\cdot a\big) \\ &\cong\{g\in\px\colon fa=ag \text{ for some }f\in\px \}\Big/ \px\cdot a.
\end{align*}
The structure of eigenrings in a free algebra over an algebraically closed field is well-understood \cite[Section 4.6]{cohn}.

\begin{thm}{\cite[Corollaries 4.6.10 and 4.6.13]{cohn}}\label{t:eigering}
Let $a\in\px$.
\begin{enumerate}[(a)]
    \item $E(a)$ is a finite-dimensional over $\kk$.
    \item If $a$ is irreducible then $E(a)=\kk$.
\end{enumerate}
\end{thm}

Factorization of homogeneous polynomials is especially simple. The following fact is well-known to experts in factorization theory and is tacitly used throughout the paper (e.g., when factorizing the highest-degree homogeneous component of a noncommutative polynomial).

\begin{prop}\label{p:notused}
Stably associated homogeneous noncommutative polynomials differ only up to a scalar multiple. 
In a complete factorization of a homogeneous noncommutative polynomial, the irreducible factors are homogeneous and unique up to scaling.
\end{prop}

\begin{proof}
For the first part, see \cite[Exercise 3.1.16]{cohn} or \cite[Remark 4.1]{HKV0}. The second part then follows from the first part and the characterization of complete factorizations in a free algebra \cite[Proposition 3.2.9]{cohn}.
\end{proof}

For later use, we record some degree bounds pertaining to coprimeness and comaximality.

\begin{lem}\label{l:comax_bds}
Let $f,g\in \px\setminus\kk$.
\begin{enumerate}[(a)]
    \item If $f$ and $g$ are stably associated, then $\deg f=\deg g$, and there is a coprime relation $fa=bg$ with $\deg a=\deg b<\deg f$.
    \item If $f$ and $g$ are right comaximal, there exist $a,b\in\px$ with $\deg a<\deg g$ and $\deg b<\deg f$ such that $fa+gb=1$. 
\end{enumerate}
\end{lem}

\begin{proof}
(a) By \cite[Propositions 2.7.4 and 2.7.6]{cohn} we have $\deg f=\deg g$. 
By Theorem \ref{t:st_assoc}, there are $a,b\in R$ such that $fa=bg$ is a coprime relation. If $\deg b\ge\deg f$, then by looking at the highest-degree homogeneous components in $fa=bg$, there exists $c\in R$ such that $\deg (b-fc)<\deg b$. Let $b'=b-fc$ and $a'=a-cg$. Clearly, $fb'=a'g$ is a coprime relation. Continuing with decreasing the degree in this manner, we obtain $a'',b''\in R$ such that $fa''=b''g$ and $\deg b''<\deg f$.

(b) Without loss of generality, let $\deg f\ge \deg g$. By \cite[Corollary 2.3.9]{cohn}, $f R \cap g R$ is a nonzero principal right ideal. By \cite[Proposition 2.8.1]{cohn} and right comaximality of $f,g$, there exist $m\in\N$, $h_1,\dots,h_{m-1}\in R\setminus\kk$, $h_m\in\R\setminus\{0\}$ and $\alpha\in\kk\setminus\{0\}$ such that
$$f=\alpha p_m(h_1,\dots,h_m),\quad g=\alpha p_{m-1}(h_1,\dots,h_{m-1}),$$
where $p_m$ is the $m$\textsuperscript{th} \emph{continuant polynomial} \cite[Section 2.7]{cohn}. By \cite[Proposition 2.7.3]{cohn}, there exist $a,b\in R$ with $\deg a<\deg g$ and $\deg b<\deg f$ such that $fa+gb=1$.
\end{proof}

\section{Rank equivalence}\label{sec:rank}

In this section, we prove that rank-equivalent polynomials are stably associated (Theorem \ref{t:rankequiv}). This result is a generalization of \cite[Corollary 2.13]{HKV} for irreducible polynomials, and of \cite[Theorem 5.2]{DKMV} for homogeneous linear matrix pencils.

The \emph{inner rank} $\rk f$ of a nonzero $d\times e$ matrix $f$ over $\px$ is the smallest $r\in\N$ such that $f=gh$ for an $d\times r$ matrix $g$ and an $r\times e$ matrix $h$ over $\px$. Equivalently \cite[Theorem 7.5.13]{cohn}, $r$ is the rank of $f$ as a matrix over the free skew field, the universal skew field of fractions of $\px$ \cite[Corollaries 2.5.2 and 7.5.14]{cohn}. 
An $d\times d$ matrix over $\px$ is \emph{full} if its inner rank equals $d$; in particular, every nonzero element of $\px$ is a full $1\times1$ matrix.
For convenience, we record the following statement, which is well-known to experts.

\begin{lem}\label{l:rankeval}
Let $f$ be a nonzero matrix over $\px$. Its inner rank equals 
$$\max_{k\in\N}\,
\frac1k \max \big\{\rk f(\uX)\colon \uX\in\mtx{k}^n\big\}.$$
\end{lem}

\begin{proof}
Let $r$ be the inner rank of $f$.
Clearly, $\rk f(\uX)\le kr$ for every $\uX\in\mtx{k}^n$.
Let $f=r\left(\begin{smallmatrix}
I_r&0\\0&0\end{smallmatrix}\right)s$ where $r,s$ are invertible matrices over the free skew field. By \cite[Proposition 2.1]{KVV10}, for some $k\in\N$ there exists $\uX\in\mtx{k}^n$ such that $r$ and $s$ are both defined and invertible at $\uX$. Then
$$f(\uX)=r(\uX)\begin{pmatrix}
I_{kr}&0\\0&0\end{pmatrix}s(\uX),$$
so $\rk f(\uX)=kr$.
\end{proof}

We start by showing that full matrices over $\px$, which are not stably associated, attain kernels of distinct dimensions when evaluated on suitable linear operators on an infinite-dimensional vector space. 
Let $f,f'$ be full matrices over $\px$; let us ad hoc say that $f'$ is a \emph{quotient} of $f$ if $f$ is stably equivalent to a matrix of the form $\left(\begin{smallmatrix}
\star&\star\\0& f'
\end{smallmatrix}\right)$.

\begin{lem}\label{l:infrank}
Let $f$ and $g$ be full matrices over $\px$. If $f$ and $g$ are not stably associated, there exists an infinite-dimensional vector space $V$ over $\kk$ and $X_1,\dots,X_n\in \End_\kk(V)$ such that
\begin{enumerate}
\item  $\ker f(\uX)$ and $\ker g(\uX)$ are finite-dimensional, with different dimensions;
\item $\dim\ker f'(\uX)\le\dim\ker f(\uX)$ and $\dim\ker g'(\uX)\le\dim\ker g(\uX)$ for all quotients $f'$ of $f$ and $g'$ of $g$.
\end{enumerate}
\end{lem}

\begin{proof}
Consider the category $\mathscr{T}$ of left $R$-modules of the form $\cM(h):=R^{1\times d}/R^{1\times d}h$ for $d\in\N$ and full $h\in R^{d\times d}$. By \cite[Proposition 3.2.1 and Theorem 3.2.3]{cohn}, $\mathscr{T}$ is an abelian category. Furthermore, $\Hom_R(M,N)$ is a finite-dimensional over $\kk$ for every $M,N\in\mathscr{T}$ by \cite[Theorem 5.8.5]{cohn}.
Since $f$ and $g$ are not stably associated, the modules $\cM(f)$ and $\cM(g)$ in $\mathscr{T}$ are not isomorphic by \cite[Corollary 0.5.5]{cohn}. Therefore, there exists $V\in\mathscr{T}$ such that $\dim_\kk \Hom_R(\cM(f),V)\neq \dim_\kk \Hom_R(\cM(g),V)$ by \cite[Theorem]{bongartz}. Let $X_1,\dots,X_n\in \End_\kk(V)$ be given by the left action of $x_1,\dots,x_n$ on $V$. If $e_1,\dots,e_d$ denotes the standard basis of the left $R$-module $R^{1\times d}$, then for every $h\in R^{d\times d}$, there is an isomorphism of $\kk$-vector spaces
$$\Hom_R(\cM(h),V)\to\ker h(\uX)\subseteq V^d,\qquad
\phi\mapsto \begin{pmatrix}\phi(e_1)\\ \vdots \\ \phi(e_d)\end{pmatrix}.$$
Therefore, $\dim_\kk \ker f(\uX)\neq \dim_\kk \ker g(\uX)$, so (1) holds. 
Next, obsrve that if $h'$ is a quotient of $h$, then the $R$-module $\cM(h')$ is a quotient of the $R$-module $\cM(h)$. Since $\Hom_R(\_,V)$ is a contravariant left-exact functor, we have $\Hom_R(\cM(h'),V)\hookrightarrow \Hom_R(\cM(h),V)$, so (2) holds.
\end{proof}

The following is a generalization of \cite[Lemma 3.1]{vol21}, and our main intermediate step towards Theorem \ref{t:rankequiv}. Together with Lemma \ref{l:rankeval}, it implies that given a full linear matrix $\Lambda$ over $\px$, every rectangular matrix tuple can be completed to a larger square tuple without increasing the kernel when evaluated at $\Lambda$.

\begin{prop}\label{p:pad}
Let $\Lambda=\sum_{i=1}^nA_ix_i\in\px^{d\times d}$ be a full linear matrix, $p\ge q$ and $T\in (\kk^{p\times q})^n$.
Assume that $\dim\ker\Lambda'(T)\le \dim\ker\Lambda(T)$ for every quotient $\Lambda'$ of $\Lambda$. 
Denote $\tilde p=p+(p-q)(d-1)$.
Then the $\tilde p d\times \tilde p d$ affine matrix
$$\cL=\sum_iA_i\otimes 
\begin{pmatrix}
\begin{matrix}T_i\\0\end{matrix} & 
\begin{matrix}
y_{i11}&\cdots&y_{i1(\tilde p-q)}\\
\vdots&\ddots&\vdots\\
y_{i\tilde p1}&\cdots&y_{i\tilde p(\tilde p-q)}\\
\end{matrix}
\end{pmatrix}
$$
in $n\tilde p(\tilde p-q)$ free variables $\uy=(y_{i\imath\jmath})_{i,\imath,\jmath}$ has inner rank $\tilde p d-\dim\ker \Lambda(T)$ over $\py$.
\end{prop}

\begin{proof}
Clearly, the inner rank of $\cL$ is at most $\tilde pd-\dim\ker \Lambda(T)$. 
By \cite[Theorem 1]{FR} it suffices to prove the following: if $U\in\kk^{\tilde pd\times a}$ and $V\in\kk^{\tilde pd\times b}$ satisfy $U^\top \cL V=0$, then $\rk U+\rk V\le \tilde pd+\dim\ker \Lambda(T)$.

After a canonical shuffle and with a slight abuse of notation, we write
$$
K=\sum_iT_i\otimes A_i,\qquad
\cL=\sum_i 
\begin{pmatrix}
\begin{matrix} T_i\\0\end{matrix} & 
\begin{matrix}
	y_{i11}&\cdots&y_{i1(p-q)}\\
	\vdots&&\vdots\\
	y_{ip1}&\cdots&y_{ip(p-q)}\\
\end{matrix}
\end{pmatrix}\otimes A_i.
$$
Let $U\in\kk^{\tilde pd\times a}$ and $V\in\kk^{\tilde pd\times b}$ be such that $U^\top \cL V=0$.
Denote
$$
U=\begin{pmatrix}U_1\\ \vdots \\U_{\tilde p}\end{pmatrix},\quad 
U_0=\begin{pmatrix}U_1\\ \vdots \\U_p\end{pmatrix},\quad 
V=\begin{pmatrix}V_1\\ \vdots \\V_{\tilde p}\end{pmatrix},\quad
V_0=\begin{pmatrix}V_1\\ \vdots \\V_q\end{pmatrix}
$$
where $U_k\in \kk^{d\times a}$ and $V_k\in \kk^{d\times b}$ for $k\ge 1$. 
In terms of these blocks, the equation $U^\top \cL V=0$ is equivalent to
\begin{align}
\label{e:constrank}U_0^\top KV_0&=0,\\
\label{e:linrank}U_k^\top  \Lambda V_\ell&=0\quad \text{for }1\le k\le \tilde p,\, q+1\le \ell\le \tilde p.
\end{align}
Since $\Lambda$ is full, 
\begin{equation}\label{e:linrank0}
\rk \begin{pmatrix}
U_1&\cdots& U_{\tilde p}
\end{pmatrix}
+\rk \begin{pmatrix}
V_{q+1}&\cdots& V_{\tilde p}
\end{pmatrix}\le d\end{equation}
by \eqref{e:linrank} and \cite[Proposition 3.1.2]{cohn} (or \cite[Theorem 1]{FR}). Denote $r=\rk \left(\begin{smallmatrix}V_{q+1}& \cdots &V_{\tilde p}\end{smallmatrix}\right)$. 
Next, we distinguish two complementary cases.
\\
(First case) Suppose 
\begin{equation}\label{e:linrank1}
\rk \begin{pmatrix}
U_{p+1}&\cdots& U_{\tilde p}
\end{pmatrix}
+\rk \begin{pmatrix}
V_{q+1}&\cdots& V_{\tilde p}
\end{pmatrix}\le d-1.
\end{equation}
By \eqref{e:constrank} and \eqref{e:linrank1},
\begin{align*}
\rk U+\rk V
&\le \rk U_0+\rk V_0+\rk \begin{pmatrix}U_{p+1}\\ \vdots \\U_{\tilde p}\end{pmatrix}
+\rk \begin{pmatrix}V_{q+1}\\ \vdots \\V_{\tilde p}\end{pmatrix}\\
&\le pd+\dim\ker K+(\tilde p-p)\rk \begin{pmatrix}U_{p+1}& \cdots &U_{\tilde p}\end{pmatrix}+(\tilde p-q)
\rk \begin{pmatrix}V_{q+1}& \cdots &V_{\tilde p}\end{pmatrix}\\
&\le pd+\dim\ker K+(\tilde p-p)(d-r-1)
+(\tilde p-q)r\\
&=\tilde pd+\dim\ker K+(p-q)r-(\tilde p-p)\\
&\le \tilde pd+\dim\ker K+(p-q)(d-1)-(\tilde p-p)\\
&= \tilde pd+\dim\ker K
\end{align*}
by the definition of $\tilde p$.
\\
(Second case) Now suppose \eqref{e:linrank1} does not hold. Then by \eqref{e:linrank0},
\begin{align*}
&\rk \begin{pmatrix}
U_1&\cdots& U_{\tilde p}
\end{pmatrix}
+\rk \begin{pmatrix}
V_{q+1}&\cdots& V_{\tilde p}
\end{pmatrix}=d,\\
&\ran \begin{pmatrix}
U_1&\cdots& U_p
\end{pmatrix}\subseteq \ran \begin{pmatrix}
U_{p+1}&\cdots& U_{\tilde p}
\end{pmatrix}.
\end{align*}
Thus, after multiplying $\Lambda$ on the left and on the right with an invertible matrix, we can assume that our matrices have the block structure
$$A_j=
\begin{pmatrix}A_j^\uparrow&A_j^\to\\0&A_j^\downarrow
\end{pmatrix},
\quad
U_k=\begin{pmatrix}
0\\ U_k^\downarrow
\end{pmatrix},\quad
V_k=\begin{pmatrix}
V_k^\uparrow\\ V_k^\downarrow
\end{pmatrix} \text{ for }k\le q,\quad
V_k=\begin{pmatrix}
V_k^\uparrow\\ 0
\end{pmatrix} \text{ for }k> q
$$
where $A_j^\uparrow$ and $V_k^\uparrow$ have $r$ rows, and $A_j^\downarrow$ and $U_k^\downarrow$ have $d-r$ rows.
Then \eqref{e:constrank} implies
$$
\begin{pmatrix}U_1^\downarrow\\ \vdots \\U_p^\downarrow\end{pmatrix}^\top 
\left(\sum_i T_i\otimes A_i^{\downarrow}\right)
\begin{pmatrix}V_1^\downarrow\\ \vdots \\V_q^\downarrow\end{pmatrix}=0,
$$
and so
$$
\rk \begin{pmatrix}U_1^\downarrow\\ \vdots\\U_p^\downarrow\end{pmatrix}+
\rk \begin{pmatrix}V_1^\downarrow\\ \vdots\\V_q^\downarrow\end{pmatrix}
\le p(d-r)+\dim\ker \left(\sum_i T_i\otimes A_i^{\downarrow}\right)
\le p(d-r)+\dim\ker K
$$
by the assumption because $\sum_i A_i^{\downarrow}x_i$ is a quotient of $\Lambda$.
Hence,
\begin{align*}
\rk U+\rk V
&\le 
\rk \begin{pmatrix}U_1^\downarrow\\ \vdots\\U_p^\downarrow\end{pmatrix}
+\rk \begin{pmatrix}V_1^\downarrow\\ \vdots\\V_q^\downarrow\end{pmatrix}
+\rk \begin{pmatrix}U_{p+1}^\downarrow\\ \vdots\\U_{\tilde p}^\downarrow\end{pmatrix}
+\rk \begin{pmatrix}V_{q+1}^\uparrow\\ \vdots\\V_{\tilde p}^\uparrow\end{pmatrix}
+\rk \begin{pmatrix}V_1^\uparrow\\ \vdots\\V_q^\uparrow\end{pmatrix}\\
&\le p(d-r)+\dim\ker K+(\tilde p-p)(d-r)+(\tilde p-q)r+qr\\
&=\tilde p d+\dim\ker K.
\end{align*}

Therefore, $\rk U+\rk V\le \tilde p d+\dim\ker K=\tilde p d+\dim\ker \Lambda(T)$ in both cases, so $\Lambda$ is full.
\end{proof}

\subsection{Rank-equivalent matrices over the free algebra}

We can now prove the main result of this section.

\begin{thm}\label{t:rankequiv}
Let $f$ and $g$ be full $c\times c$ matrices over $\px$. Then $f$ and $g$ are rank-equivalent if and only if $f$ and $g$ are stably associated.
\end{thm}

\begin{proof}
The implication $(\Leftarrow)$ is clear by the definition of stable association. To prove $(\Rightarrow)$, assume $f$ and $g$ are not stably associated. 
By \cite[Theorem 5.8.3]{cohn}, there is $d\in\N$ such that $f$ and $g$ are stably associated to a full affine matrices $\Lambda^1=A_0^1+\sum_iA_i^1x_i$ and $\Lambda^2=A_0^2+\sum_iA_i^2x_i$, respectively, with $A_i^1,A_i^2\in\mtx{d}$. 
By Lemma \ref{l:infrank}, there is a tuple of operators $\uX$ on an infinite-dimensional space $V$ such that $k_j=\dim\ker \Lambda^j(\uX)$ are finite, $k_1\neq k_2$, and $\dim\ker {\Lambda^j}'(\uX)\le k_j$ for quotients ${\Lambda^j}'$ of $\Lambda^j$. 
Let $U$ be a finite-dimensional subspace of $V$ such that $\ker\Lambda^1(\uX)+\ker\Lambda^2(\uX)\subseteq \kk^d\otimes U$. 
Let $\widehat{U}$ be a finite-dimensional subspace of $V$ such that $\widehat{U}\supseteq \sum_i X_iU$ and $q:=\dim U\le \dim \widehat{U}=:p$. 
Let us view $T_i=X_i|_U:U\to \widehat{U}$ for $i=1,\dots,n$ as $p\times q$ matrices; then
$$\dim\ker \sum_{i=0}^n A_i^1\otimes T_i\neq 
\dim\ker \sum_{i=0}^n A_i^2\otimes T_i,$$
where $T_0=\left(\begin{smallmatrix}I\\0\end{smallmatrix}\right)$. 
Note that the assumption of Proposition \ref{p:pad} is satisfied; thus, the $\tilde pd\times \tilde pd$ affine matrix (with $\tilde p=p+(p-q)(d-1)$)
$$\sum_{i=0}^n A_i^j\otimes 
\begin{pmatrix}
\begin{matrix} T_i\\0\end{matrix} & 
\begin{matrix}
y_{i11}&\cdots&y_{i1(p-q)}\\
\vdots&\ddots&\vdots\\
y_{i\tilde p1}&\cdots&y_{i\tilde p(\tilde p-q)}
\end{matrix}
\end{pmatrix}
$$
ha inner rank $\tilde pd-k_j$. By Lemma \ref{l:rankeval} there exists $\uY\in\mtx{\ell}^{(n+1)\tilde p(\tilde p-q)}$ such that
\begin{equation}\label{e:bigrank}
\rk
\sum_{i=0}^nA_i^j\otimes 
\begin{pmatrix}
\begin{matrix}T_i\otimes I_\ell\\0\end{matrix} & 
\begin{matrix}
Y_{i11}&\cdots&Y_{i1(\tilde p-q)}\\
\vdots&\ddots&\vdots\\
Y_{i\tilde p1}&\cdots&Y_{i\tilde p(\tilde p-q)}
\end{matrix}
\end{pmatrix}
=(\tilde pd-k_j)\ell.
\end{equation}
Since this holds for $\uY$ in a Zariski open subset of $\mtx{\ell}^{(n+1)\tilde p(\tilde p-q)}$, we can furthermore assume that the $(\tilde p-q)\times(\tilde p-q)$ matrix
$$\begin{pmatrix}
Y_{0(q+1)1}&\cdots&Y_{0(q+1)(\tilde p-q)}\\
\vdots&\ddots&\vdots\\
Y_{0\tilde p1}&\cdots&Y_{0\tilde p(\tilde p-q)}
\end{pmatrix}$$
is invertible. Therefore, the $\tilde p\ell\times \tilde p\ell$ matrix
$$W=\begin{pmatrix}
\begin{matrix}T_0\otimes I_\ell \\0\end{matrix} & 
\begin{matrix}
Y_{011}&\cdots&Y_{01(\tilde p-q)}\\
\vdots&\ddots&\vdots\\
Y_{0\tilde p1}&\cdots&Y_{0\tilde p(\tilde p-q)}
\end{matrix}
\end{pmatrix}$$
is invertible. 
Let $\uZ\in\mtx{\tilde p\ell}^n$ be given as
$$Z_j=W^{-1}\begin{pmatrix}
\begin{matrix}T_i\otimes I_\ell\\0\end{matrix} & 
\begin{matrix}
Y_{i11}&\cdots&Y_{i1(\tilde p-q)}\\
\vdots&\ddots&\vdots\\
Y_{i\tilde p1}&\cdots&Y_{i\tilde p(\tilde p-q)}
\end{matrix}\end{pmatrix}$$
for $j=1,\dots,n$.
By \eqref{e:bigrank},
\begin{align*}
\rk f(\uZ)
&=c\tilde p\ell-\dim\ker f(\uZ)=d\tilde p\ell-(\tilde pd-k_1)\ell=k_1\ell\\
&\neq k_2\ell=d\tilde p\ell-(\tilde pd-k_2)\ell=c\tilde p\ell-\dim\ker g(\uZ)
=\rk g(\uZ),
\end{align*}
as desired.
\end{proof}

Theorem \ref{t:rankequiv} and \cite[Exercise 3.1.16]{cohn} (or \cite[Remark 4.1]{HKV0}) imply the following corollary for homogeneous polynomials. Nevertheless, we provide an alternative self-contained proof.

\begin{cor}
Let $f,g\in\px$ be homogeneous. Then $f,g$ are stably associated if and only if $g=\lambda f$ for some nonzero $\lambda\in\kk$.
\end{cor}

\newcommand{\fm}{\mathfrak{m}}

\begin{proof}
Let $J_f$ and $J_g$ be the two-sided ideals in $\px$ generated by $f$ and $g$, respectively.
For $d =\max\{\deg f, \deg g\}+1$ let $\fm_d$ be the ideal generated by all the monomials of degree $d$.
Consider the finite-dimensional $\kk$-algebra $\cA = \px/(J_g + \fm_d)$. For $h \in \px$, let $L_h$ be the operator of left multiplication by $h$ on $\cA$. We note that $L_h = h(L_{x_1},\ldots,L_{x_d})$. Hence, in particular, $L_g = 0$. Since the polynomials $f$ and $g$ are rank-equivalent, $L_f = 0$.
This means that $f \in J_g + \fm_d$. However, since $f$ and $g$ are both homogeneous and $d$ is larger than their degrees, we have $f \in J_g$. Analogously, $g \in J_f$, and thus $J_f = J_g$. This immediately implies that $g = \lambda f$ for some $\lambda\in\kk\setminus\{0\}$.
\end{proof}

\subsection{Relation with joint similarity of matrices}

Let us draw a parallel with the main result of \cite{DKMV} on joint similarity of matrix tuples and ranks of corresponding linear pencils.
By \cite[Theorem 1.1]{DKMV}, two tuples of $c\times c$ matrices $(A_1,\dots,A_n)$ and $(B_1,\dots,B_n)$ are jointly similar (meaning that there is $P\in\GL_c(\kk)$ such that $B_j=PA_jP^{-1}$ for all $j$) if and only if
\begin{equation}\label{e:derksen1}
\rk\big(I\otimes X_0+A_1\otimes X_1+\cdots+A_n\otimes X_n\big)=
\rk\big(I\otimes X_0+B_1\otimes X_1+\cdots+B_n\otimes X_n\big)
\end{equation}
for all tuples of square matrices $\uX=(X_0,\dots,X_n)$ (moreover, an upper bound on the dimension of $\uX$ is given). 
Theorem \ref{t:rankequiv} recovers this statement (without the dimension bound). 
Indeed, if $Ix_0+\sum_jA_jx_j$ and $Ix_0+\sum_jB_jx_j$ are rank-equivalent, then they are stably associated by Theorem \ref{t:rankequiv}. Since they are full and homogeneous, they are conjugated by a constant matrix by \cite[Theorem 5.8.3]{cohn}.

In general, $\uA$ and $\uB$ are not jointly similar if \eqref{e:derksen1} holds only when $X_0=I$ (e.g., take $n=1$, $A_1=\left(\begin{smallmatrix}
0&1\\0&0\end{smallmatrix}\right)$ and $B_1=\left(\begin{smallmatrix}
0&0\\0&0\end{smallmatrix}\right)$).
The following consequence of Theorem \ref{t:rankequiv} examines the effect of restricting to $X_0=I$ in \eqref{e:derksen1}.

\begin{cor}
Let $\uA=(A_1,\dots,A_n),\uB=(B_1,\dots,B_n)\in\mtx{c}^n$.
Then
\begin{equation}\label{e:derksen2}
\rk\big(I\otimes I+A_1\otimes X_1+\cdots+A_n\otimes X_n\big)=
\rk\big(I\otimes I+B_1\otimes X_1+\cdots+B_n\otimes X_n\big)
\end{equation}
for all $\uX=(X_1,\dots,X_n)$ if and only if  $I+\sum_jA_jx_j$ and $I+\sum_jB_jx_j$ are stably associated.

Furthermore, if
$$\bigcap_j \ker A_j=\bigcap_j \ker A_j^\top =
\bigcap_j \ker B_j=\bigcap_j \ker B_j^\top =\{0\},$$
then \eqref{e:derksen2} holds for all $\uX$ if and only if $\uA$ and $\uB$ are jointly similar.
\end{cor}

\begin{proof}
The first part is an immediate consequence of Theorem \ref{t:rankequiv} since $F=I+\sum_jA_jx_j$ and $G=I+\sum_jB_jx_j$ are full.
By \cite[Theorem 5.8.3]{cohn} and the kernel assumption on $\uA$ an $\uB$, the affine matrices $F$ and $G$ are stably associated if and only if and only if $GQ=PF$ for some $P,Q\in\GL_c(\kk)$, in which case $Q=P$ and $B_j=PA_jP^{-1}$.
\end{proof}

\subsection{Stable association of powers}

If $f,g\in\px$ are rank-equivalent, it is not necessarily true that their powers are likewise rank-equivalent (that is, while $f$ and $g$ may always have the same number of Jordan blocks at 0, these blocks might not have the same sizes). 
For example, let $f=xyxy+xy+x$ and $g=xy^2x+xy+x$. Then $f$ and $g$ are stably associated, and therefore rank-equivalent, by $f(yx+1)=(xy+1)g$ and Theorem \ref{t:st_assoc}.
On the other hand, $f^2$ and $g^2$ are not rank-equivalent, and thus also not stably associated by Theorem \ref{t:rankequiv}. Indeed, denote
$$X=\begin{pmatrix}1&0\\0&0\end{pmatrix},\quad 
Y=\begin{pmatrix}0&1\\-1&0\end{pmatrix}.$$
Then $f(X,Y)^2\neq0$ and $g(X,Y)^2=0$.

On the other hand, the converse of the statement holds (and is useful in Section \ref{sec:sim} below).

\begin{lem}\label{l:power}
Let $f,g\in\px$. If $f^k$ and $g^k$ are stably associated for some $k\in\N$, then $f$ and $g$ are stably associated.
\end{lem}

\begin{proof}
Since $f^k$ and $g^k$ are stably associated, there is an isomorphism of left $R$-modules $\phi:R/Rf^k\to R/Rg^k$. Consider the chain of submodules
$$R/Rf^k \supset Rf/Rf^k\supset \cdots\supset  Rf^{k-1}/Rf^k\supset \{0\},$$
where the quotient of any two consecutive terms is isomorphic to $R/Rf$. The image of this chain in $R/Rg^k$,
$$\phi(R/Rf^k) \supset \phi(Rf/Rf^k)\supset \cdots\supset  \phi(Rf^{k-1}/Rf^k)\supset \{0\},$$
corresponds to a factorization $g^k=f_1\cdots f_k$ where each $f_j$ is stably associated to $f$ by \cite[Section 3.2]{cohn}. If two polynomials are stably associated, then their irreducible factors can be grouped into stably associated pairs. Thus, $g,f_1,\dots,f_k,f$ all have the same irreducible factors up to stable association (and ordering in a complete factorization). Let $\ell$ be the highest common left factor of $g$ and $f_1$. If $g$ and $f_1$ are scalar multiples of $\ell$, then $g$ and $f$ are stably associated, as desired.

Suppose that $g=\ell g'$ and $f_1=\ell f_1'$ for nonconstant $g',f_1'\in R$. Let $\ell'$ be an irreducible left factor of $g'$, and write $g'=\ell' g''$. Then $\ell'$ is stably associated to an irreducible factor $\ell''$ of $f_1'$, and write $f_1'=a\ell'' b$. The relation
$$ \ell'\cdot (g''g^{k-1})=\ell^{-1}g^k=\ell^{-1}f_1\cdots f_k = a\ell''\cdot (bf_1' f_2\cdots f_k)$$
shows that $\ell' R\cap (a\ell'')R\neq\{0\}$. Since $\ell'$ and $\ell''$ are stably associated, \cite[Corollary 4.3.4 and Theorem 4.2.3(a)]{cohn} implies that $a\ell''\in \ell' R$. Thus, $\ell \ell'$ is a common left factor of $g$ and $f_1$, contradicting the choice of $\ell$.
\end{proof}

\section{Isospectrality}\label{sec:eig}

We say that $f,g\in\px$ are \emph{intertwined} if there exists a nonzero $a\in\px$ (called an intertwiner) such that $fa=ag$. By \cite[Corollary 4.4]{vol20}, two polynomials are isospectral if and only if they are intertwined. In this section, we revisit and strengthen this result, provide a new characterization of isospectrality in terms of elementary intertwinedness, and derive some consequences of isospectrality that are used in the next section.

\subsection{Elementary intertwinedness and isospectrality}

Let $\lambda\in\kk$ and $a,b\in\px$. We say that noncommutative polynomials $\lambda+ab$ and $\lambda+ba$ are \emph{elementary intertwined}. Note that $(\lambda+ab)a=a(\lambda+ba)$, so elementary intertwinedness implies intertwinedness. 

\begin{thm}\label{t:isospec}
The following are equivalent for $f,g\in\px$:
\begin{enumerate}[(i)]
\item for infinitely many $k\in\N$ there is a Zariski dense set $\cD_k\subseteq\mtx{k}^n$ such that $f(\uX)$ and $g(\uX)$ share an eigenvalue for every $\uX\in \cD_k$;
\item $f$ and $g$ are isospectral;
\item $f$ and $g$ are intertwined;
\item there exist $\ell\in\N$ and $f_1,\dots,f_{\ell}\in \px$ with $f_1=f$ and $f_{\ell}=g$ such that $f_i$ and $f_{i+1}$ are elementary intertwined for all $1\le i\le \ell-1$.
\end{enumerate}
\end{thm}

\begin{proof}
The equivalence (ii)$\Leftrightarrow$(iii) is \cite[Corollary 4.4]{vol20}.
The implication (ii)$\Rightarrow$(i) is trivial, and the implication (iv)$\Rightarrow$(ii) follows by (iii)$\Leftrightarrow$(ii) and transitivity of isospectrality.

(iii)$\Rightarrow$(iv): Let nonzero $a\in R$ be such that $fa=ag$. 
We prove the statement by induction on the number of irreducible factors of $a$.

If $a$ is irreducible, then $E(a)=\kk$ by Theorem \ref{t:eigering}(b). Therefore $f=\lambda+ab$ and $g=\lambda+ba$ for some $\lambda\in\kk$ and $b\in R$.

Now suppose that $a$ is not irreducible. Note that $g$ represents an element of $E(a)$. Since $E(a)$ is finite-dimensional by Theorem \ref{t:eigering}(a), there exists $\lambda\in\kk$ such that $g-\lambda$ is a zero divisor in $E(a)$. Thus there exists $c\in R\setminus R a$ such that $ac\in R a$ and $(g-\lambda)c\in R a$.
If $1\in R(g-\lambda)+R a$, then
$c\in R(g-\lambda)c+R a c\subseteq R a$, a contradiction. Therefore, $R(g-\lambda)+R a$ is a proper left ideal in $R$. Since $0\neq (f-\lambda)a=a(g-\lambda)\in R(g-\lambda)\cap R a$, it follows by \cite[Theorem 2.3.7]{cohn} that $R(g-\lambda)+R a = R a_1$ for some nonconstant $a_1\in R$. Therefore $g-\lambda=ba_1$ and $a=\tilde{a}a_1$ for some $b,\tilde{a}\in R$. Let $\tilde{g}=\lambda+a_1b$. Then $g$ and $\tilde{g}$ are elementary intertwined, and
$$f\tilde{a}a_1=fa=ag
=\tilde{a}a_1(\lambda+ba_1)
=\tilde{a}(\lambda+a_1b)a_1
=\tilde{a}\tilde{g}a_1,$$
so $f\tilde{a}=\tilde{a}\tilde{g}$. Since $\tilde{a}$ has less irreducible factors than $a$, there is a sequence of elementary intertwined polynomials between $f$ and $\tilde{g}$ by the induction hypothesis. Hence, there is a sequence of elementary intertwined polynomials between $f$ and $g$.

(i)$\Rightarrow$(ii): The implication is trivial if either $f$ or $g$ is constant. Thus, assume $f,g\in R\setminus\kk$.
Let $\cO_k$ denote the set of $\uX\in\mtx{k}^n$ such that $f(\uX)$ and $g(\uX)$ each have $n$ pairwise distinct eigenvalues. For all large enough $k\in\N$, $\cO_k$ is a nonempty Zariski open subset of $\mtx{k}^n$ by \cite[Corollary 2.10]{BreVol}. By the assumption, $\cD_k\cap\cO_k$ is Zariski dense in $\mtx{k}^n$ for infinitely many $k\in\N$. For such $k$ and $\uX\in \cD_k\cap\cO_k$, the characteristic polynomials of $f(\uX)$ and $g(\uX)$ each have pairwise distinct roots, and share a root. Let $\underline{\Omega}$ denote the $n$-tuple of $k\times k$ generic matrices with independent indeterminate entries, $\operatorname{UD}_k$ the universal division $\kk$-algebra generated by $\underline{\Omega}$ \cite[Section 3.2]{Row}, and $\cZ$ the center of $\operatorname{UD}_k$. Consider monic polynomials $\chi_f=\det(tI-f(\underline{\Omega}))$ and $\chi_g=\det(tI-g(\underline{\Omega}))$ in $\cZ[t]$. Since $\cO_k$ is dense in $\mtx{k}^n$, the polynomials $\chi_f$ and $\chi_g$ each have pairwise distinct roots in $\overline{\cZ}$, and are consequently minimal polynomials of $f(\underline{\Omega})$ and $g(\underline{\Omega})$. Thus, $\chi_f$ and $\chi_g$ are irreducible over $\cZ$ since $\cZ[t]/(\chi_f)$ and $\cZ[t]/(\chi_g)$ are isomorphic to the $\cZ$-subfields in $\operatorname{UD}_k$ generated by $f(\underline{\Omega})$ and $g(\underline{\Omega})$, respectively. Since $\cD_k$ is dense in $\mtx{k}^n$, the polynomials $\chi_f$ and $\chi_g$ share a root in $\overline{\cZ}$. By the irreducibility, $\chi_f=\chi_g$. Thus, the spectra of $f(\uX)$ and $g(\uX)$ coincide for all $\uX\in\mtx{k}^n$, for infinitely many $k\in\N$. Hence, $f$ and $g$ are isospectral.
\end{proof}

Theorem \ref{t:isospec} in particular shows that intertwinedness coincides with the transitive closure of elementary intertwinedness (for an analogous phenomenon in symbolic dynamics, see elementary and strong shift equivalence of integer matrices \cite[Section 7.2]{LindMarcus}). 
Characterization of isospectrality in terms of elementary intertwinedness allows one (at least in principle) to construct the isospectrality equivalency class of a noncommutative polynomial (see Example \ref{exa:long}).
Item (i) of Theorem \ref{t:isospec} is used later in Section \ref{sec:real} to investigate the pointwise norm equality.
For homogeneous polynomials, isospectrality implies elementary intertwinedness, as follows.

\begin{cor}
Let $f,g\in\px$ be homogeneous. Then $f,g$ are isospectral if and only if $f=ab$, $g=ba$ for some homogeneous $a,b\in\px$.
\end{cor}

\begin{proof}
If $f,g$ are isospectral, then $fa=ag$ for some nonzero $a\in\px$ by Theorem \ref{t:isospec}. Since $f$ and $g$ are homogeneous, we may take $a$ homogeneous (more precisely, we can replace the initial $a$ with any of its nonzero homogeneous components). We distinguish two cases.
If $\deg a\le \deg f$, then the uniqueness of factorization for homogeneous polynomials as in Proposition \ref{p:notused} applied to $fa=ag$ shows that $f=ab$ for some $b\in\px$, and then $g=ba$. 
If $\deg a>\deg f$, then Proposition \ref{p:notused} implies $a=fa'$ for some $a'\in\px$, and consequently $fa'=a'g$. Continuing in this fashion, we eventually arrive at the first case.
\end{proof}

\begin{rem}\label{r:unique}
An intertwiner of $f$ and $g$ (if it exists) of minimal degree is unique up to scaling, and all other intertwiners are its left multiples by polynomials that commute with $f$. Indeed, let $M=\{a\in\px\colon fa=ag\}$, and let $\cC$ be the centralizer of $f$ in $\px$. By Bergman's centralizer theorem \cite{ber69}, the $\kk$-algebra $\cC$ is isomorphic to the univariate polynomial ring over $\kk$; furthermore, the centralizer $\widetilde{C}$ of $f$ in the universal skew field of fractions of $f$ equals the field of fractions of $\cC$. Let us view $M$ as a left $\cC$-module. Clearly, $M$ is torsion-free; since $\cC$ is a principal ideal domain, $M$ is a free $\cC$-module. If $0\neq a,b\in M$, then $fba^{-1}=bga^{-1}=ba^{-1}f$, and so $ba^{-1}\in\widetilde{\cC}$. Thus, the vector space $\widetilde{\cC}\otimes_{\cC} M$ is one-dimensional, and so $M$ is free of rank 1. Hence, $M=\cC\cdot a$ for an intertwiner $a$ of minimal degree.
\end{rem}

\begin{exa}\label{exa:long}
Let $p,q,r\in \kk[x]$ and $p,q\neq0$. Consider
$$f=pyq+r,\qquad g=qyp+r$$
in $\kk\!\Langle x,y\Rangle$. Note that $qfp=pgq$. If $X$ is such that $p(X),q(X)$ are invertible, then $$g(X,Y)=\left(p(X)q(X)^{-1}\right)^{-1}f(X,Y)\left(p(X)q(X)^{-1}\right).$$
In particular, $f(X,Y)$ and $g(X,Y)$ are similar for a generic pair $(X,Y)$, and thus $f$ and $g$ are isospectral by the continuity of eigenvalues. 
Thus $f$ and $g$ are intertwined by Theorem \ref{t:isospec}.
If $p\neq q$ and $\deg p,\deg q\ge1$, then $f$ and $g$ are not elementary intertwined. 

The reader may observe that for general $p,q,r$, finding a nonzero $a\in \kk\!\Langle x,y \Rangle$ such that $fa=ag$ or $ga=af$ is not completely straightforward. 
For a case study, fix a finite subset $S\subset\kk$. For every $A\subseteq S$ denote
$p_A=\prod_{\alpha\in A}(x-\alpha)\in\kk[x]$.
Then
$$\left\{p_Ayp_{S\setminus A}+x\colon A\subseteq S\right\}$$
is an equivalence class (of size $2^{|S|}$) for the transitive extension of the elementary intertwinedness, and thus for isospectrality. 
Indeed, observe that $p_Ayp_{S\setminus A}+x-\lambda$ for $\lambda\in\kk$ is either irreducible, or has two irreducible factors: $x-\lambda$ and a factor of degree $|S|$. This shows that $p_Ayp_{S\setminus A}+x$ is elementary intertwined only with $p_Byp_{S\setminus B}+x$ for $|S\setminus (A\cap B)|\le1$.
More generally, one needs $|S\setminus(A\cap B))|$ elementary intertwined pairs to pass from $p_Ayp_{S\setminus A}+x$ to $q_Byp_{S\setminus B}+x$.
In particular, to pass from $f=p_Sy+x$ to $g=yp_S+x$, one requires $|S|$ elementary intertwined pairs. Concretely, if $S=\{1,\dots,s\}$,
\begin{equation}\label{e:chain}
p_{\{1,\dots,s\}}y+x
\rightsquigarrow p_{\{2,\dots,s\}}yp_{\{1\}}+x
\rightsquigarrow p_{\{3,\dots,s\}}yp_{\{1,2\}}+x
\rightsquigarrow\cdots\rightsquigarrow
yp_{\{1,\dots,s\}}+x.
\end{equation}
Note that $\deg f=\deg g=|S|+1$.
The minimal $a\neq 0$ such that $fa=ag$ has degree $|S|$; for example, $a=p_S$.
On the other hand, as the proof of Theorem \ref{t:isospec}(iii)$\Rightarrow$(iv) indicates, the minimal $a\neq 0$ such that $ga=af$ has degree $|S|^2$. Indeed, retracing \eqref{e:chain} using the identity $(v(u-\lambda)+u)(v+1)=(v+1)((u-\lambda)v+u)$ for a scalar $\lambda$, one sees that
$$a=\big(yp_{\{1,\dots,s-1\}}+1\big)
\big(p_{\{s\}}yp_{\{1,\dots,s-2\}}+1\big)
\big(p_{\{s-1,s\}}yp_{\{1,\dots,s-3\}}+1\big)\cdots
\big(p_{\{2,\dots,s\}}y+1\big)
$$
satisfies $ga=af$.
\end{exa}

\subsection{Operator isospectrality}

Let $V$ be a vector space over $\kk$. The spectrum of a linear operator $Y:V\to V$ is the set $\sigma(Y)=\{\lambda\in\kk\colon Y-\lambda I \text{ is not invertible}\}$.
We say that $f,g\in\px$ are \emph{operator isospectral} if the spectra of $f(\uX)$ and $g(\uX)$ coincide, for all tuples of linear operators $V\to V$, and all vector spaces $V$.\footnote{
It may be natural to restrict to $\kk=\C$, and evaluations on tuples of bounded operators on a separable Banach/Hilbert space. However, the conclusions of this subsection would remain the same.
} 
Clearly, operator isospectrality implies isospectrality. However, the converse fails, as the following example shows.

\begin{exa} \label{ex:not_op_isosp}
Let $f = xy$ and $g = yx$. Then $f$ and $g$ are elementary intertwined, and thus isospectral. 
Let $S \in \cB(\ell^2(\N))$ be the unilateral shift. Then, $f(S,S^*)$ is a non-trivial projection and $g(S,S^*) = I$. Therefore, $f$ and $g$ are not operator isospectral.
\end{exa}

While operator isospectrality is stronger than isospectrality, it is still weaker than pointwise similarity. More precisely, it does not imply rank equivalence, as demonstrated in Example \eqref{ex:opspec} below.
First, we record a few straightforward observations regarding spectra of noncommutative polynomials on operators.

\begin{lem}\label{l:easy}
Let $f,g\in\px$, and let $\uX$ be a tuple of linear operators on a vector space over $\kk$.
\begin{enumerate}[(a)]
    \item If $f,g$ are stably associated, then $f(\uX)$ is invertible if and only if $g(\uX)$ is invertible.
    \item If $f,g$ are elementary intertwined as $f=\lambda+ab$ and $g=\lambda+ba$, then $\sigma (f(\uX))\setminus\{\lambda\} = \sigma (g(\uX))\setminus\{\lambda\}$.
    \item If $f,g$ are isospectral, there exists a finite set $S\subset\kk$, independent of $\uX$, such that $\sigma(f(\uX))\setminus S=\sigma(g(\uX))\setminus S$.
\end{enumerate}
\end{lem}

\begin{proof}
(a) Follows directly from the definition of stable association, using that $Y$ is invertible if and only if $Y\oplus I$ is invertible.

(b) If $\beta\neq \lambda$, then $f-\beta$ and $g-\beta$ are stably associated, via the comaximal relation $(f-\beta)a=a(g-\beta)$. Then, (b) follows from (a).

(c) If $f,g$ are isospectral, one can pass from one to the other by a finite chain of elementary intertwined pairs by Theorem \ref{t:isospec}. Then, (c) follows from (b).
\end{proof}

\begin{exa}\label{ex:opspec}
Let $f = x y^2 x$ and $g = y x^2 y$. 
Let $X, Y$ be linear operators on a vector space over $\kk$.
By Lemma \ref{l:easy}(b), $\sigma(f(X,Y))\setminus \{0\}=\sigma(g(X,Y))\setminus \{0\}$. 
Thus, we only need to investigate the membership of $0$ in the spectra, which is only possible if at least one of $X,Y$ is singular (non-invertible).
If $X$ is not injective, then $f(X,Y)$ is not injective; if $X$ is not surjective, then $f(X,Y)$ is not surjective. 
We conclude that if $X$ is singular, then $0 \in \sigma(f(X,Y))$.
Similarly, if $Y$ is singular, then $0 \in \sigma(g(X,Y))$. Therefore, if both $X$ and $Y$ are singular, then the spectra of $f(X,Y)$ and $g(X,Y)$ coincide. We are left with the option that $X$ is singular and $Y$ is not, or vice versa. If $X$ is singular, then so is $f(X,Y)$. In order for the spectra to differ, $g(X,Y)$ must be invertible. However, if both $g(X,Y)$ and $Y$ are invertible, then so is $X^2$, which is impossible. Thus, the spectra of $f(X,Y)$ and $g(X,Y)$ always coincide.

To see that $f$ and $g$ are not rank-equivalent, let
\[
X = \begin{pmatrix} 0 & 1 \\ 0 & 0 \end{pmatrix} 
\quad\text{and}\quad
Y = \frac{1}{2} \begin{pmatrix} 1 & 1 \\ 1 & 1 \end{pmatrix}.
\]
On the one hand, $X^2 = 0$ implies $g(X,Y) = 0$. On the other hand, $Y$ is a projection, and $f(X,Y) = XYX = \frac{1}{2} X \neq 0$.
\end{exa}

The above reasoning extends to the following family of non-trivial operator isospectral pairs: 
if $a_1,a_2\in\px$ are stably associated and $b_1,b_2\in\px$ are stably associated, then
$a_1b_1b_2a_2$ and $b_2a_2a_1b_1$ are operator isospectral polynomials. 
Similarly, if $a,b\in\px$ are such that $ab$ and $ba$ are stably associated, then $ab$ and $ba$ are operator isospectral (see Example \ref{exa:unexpected} below for an instance of noncommuting $a,b$ with this feature). We leave the quest for characterization of operator isospectrality and related problems for future work.

\subsection{Isospectrality and composition}

The following observations about isospectrality and composition are needed later in Section \ref{sec:sim} below.

\begin{lem}\label{l:isospecfactors}
Let $f,g\in\px$ be isospectral, and $p\in\kk[t]$. Then every irreducible factor of $p(f)$ is stably associated to a factor of $p(g)$.
\end{lem}

\begin{proof}
Given $h\in\px$, its \emph{free locus} $\cZ(h)$ \cite{HKV} is the set of all matrix tuples $\uX$ such that $\det h(\uX)=0$. Since $f$ and $g$ are isospectral, so are $p(f)$ and $p(g)$, and in particular $\cZ(p(f))=\cZ(p(g))$. By \cite[Theorem 2.12]{HKV} it then follows that every irreducible factor of $p(f)$ is stably associated to a factor of $p(g)$.
\end{proof}

\begin{lem}\label{l:twoeig}
Let $f\in\px$, and let $\alpha,\beta\in\kk$ be distinct. Then no factor of $f-\alpha$ is stably associated to a factor of $f-\beta$.
\end{lem}

\begin{proof}
Suppose $f-\alpha=ahb$ and $f-\beta=a'h'b'$, where $h$ and $h'$ are stably associated. Then
$$ahb-a'h'b'=\beta-\alpha.$$
But such a relation cannot hold in $\px$ by \cite[Theorem 4.2.3(f)]{cohn} because $\px$ possesses the strong distributive factor lattice property \cite[Corollary 4.3.4]{cohn}.
\end{proof}

We say that $f\in\px$ is \emph{composite} if $f=p(\tilde f)$ for some $\tilde f\in\px$ and a nonlinear $p\in\kk[t]$, and \emph{non-composite} otherwise.
Observe that detecting composition is straightforward by Bergman's centralizer theorem \cite{ber69}: namely, $f$ is composite if and only if the linear space $\{h\in\px\colon fh=hf\ \&\ \deg h<\deg f\}$ is nonzero.
The following assertion appears implicitly in \cite{vol20}. 

\begin{lem} \label{lem:isospec_non_comp}
If $f, g \in \px$ are isospectral, there exist $p\in\kk[t]$ and non-composite isospectral $\tilde f,\tilde g\in\px$ such that $f=p(\tilde f)$ and $g=p(\tilde g)$.
\end{lem}

\begin{proof}
Write $f=p(\tilde f)$ and $g=q(\tilde g)$ for monic $p,q\in\kk[t]$ and non-composite $\tilde{f},\tilde{g}\in R$. By \cite[Theorem 3.2]{vol20}, there is $A\subseteq\kk$ such that $\kk\setminus A$ is finite and $\tilde f-\alpha,\tilde g-\alpha$ is irreducible for all $\alpha\in A$. 
Since $\kk\setminus A$ is finite, there exists $\Lambda\subset\kk$ such that $\kk\setminus\Lambda$ is finite and $p-\lambda,q-\lambda$ both have pairwise distinct roots in $\kk\setminus A$, for all $\lambda\in\Lambda$. Let
$$p-\lambda=\prod_{j=1}^r(t-\beta_j),\quad q-\lambda=\prod_{j=1}^s(t-\gamma_j).$$
Since $\tilde f-\beta_j,\tilde g-\gamma_j$ are irreducible, Lemmas \ref{l:twoeig} and \ref{l:isospecfactors} imply that $r=s$, and there is a permutation $\pi$ of $\{1,\dots,r\}$ such that $f-\beta_j$ is stably associated to $g-\gamma_{\pi(j)}$ for $j=1,\dots,n$. In particular, after reordering we have $\gamma_j=\beta_j+g(0)-f(0)$ for all $j=1,\dots,n$. Hence, $p=q$ after an affine transformation, and $\tilde f-\beta_j,\tilde g-\beta_j$ are stably associated for $j=1,\dots,n$. Furthermore, as $\lambda\in\Lambda$ varies, the roots $\beta_j$ attain all but finitely many values in $\kk$. Thus, $\tilde f-\beta,\tilde g-\beta$ are stably associated (in particular, rank-equivalent) for all but finitely many $\beta\in\kk$. Hence, $\tilde f$ and $\tilde g$ are isospectral by the continuity of eigenvalues.
\end{proof}

\begin{rem}
Lemma \ref{lem:isospec_non_comp} shows that deciding isospectrality reduces to non-composite polynomials. By the proof of \cite[Corollary 4.4]{vol20}, non-composite $f,g\in\px$ are isospectral if and only if $(f-t)a=b(g-t)$ for some nonzero $a,b\in \kk(t)\otimes\px$ of degree less than $\deg f$, which turns deciding isospectrality into a linear problem over the rational field $\kk(t)$.
\end{rem}

\section{Pointwise similarity}\label{sec:sim}

In this section, we show that two polynomials are pointwise similar if and only if they are equal (Theorem \ref{t:similar}).
Before embarking on the proof, let us briefly comment on the perceived complexity of this problem.\footnote{
The authors admit that they expected a quicker resolution.
} 
On the one hand, pointwise similarity is a much more restrictive local equivalence than isospectrality and rank equivalence, indicating a likelihood for a strong algebraic characterization. On the other hand, pointwise similarity lacks some desired features of isospectrality (continuity) and rank equivalence (clear interpretation of counting homomorphisms between free algebra modules).
The first main ingredient of the resolution of this problem comes from the Jordan canonical form: $f$ and $g$ are pointwise similar if and only if $(f-\lambda)^k$ and $(g-\lambda)^k$ are stably associated, for all $\lambda\in\kk$ and $k\in\N$. However, as stable association of powers turns out to be inconvenient to work with, the intertwining emerging from isospectrality is the second main ingredient.

\begin{prop}\label{p:irr_similar}
Let $f,g\in\px$. Assume that one of the following holds:
\begin{enumerate}[(a)]
    \item $f-\lambda$ is irreducible and stably associated to $g-\lambda$ for every $\lambda\in\kk$;
    \item there is $a\in\px$ such that $(f-\lambda)a=a(g-\lambda)$ is coprime for every $\lambda$.
\end{enumerate}
Then $f=g$.
\end{prop}

\begin{proof}
Assume (a) holds. Let $\lambda\in\kk$ be arbitrary. Since $f-\lambda$ and $g-\lambda$ are stably associated, $g-\lambda$ is also irreducible. 
Consider the system of linear equations over $\kk[t]$,
\begin{equation}\label{e:linsys}
(f-t)A(t)=B(t)(g-t),\qquad \deg A(t),\deg B(t)<\deg f
\end{equation}
in coefficients of $A(t),B(t)\in\kk[t]\otimes R$. By stable association and Lemma \ref{l:comax_bds}(a), for every $\lambda\in\kk$, the specialization of the linear system \eqref{e:linsys} at $t=\lambda$ admits a nonzero solution.
Therefore, the system \eqref{e:linsys} itself admits a nonzero solution $A_0(t),B_0(t)\in\kk[t]\otimes R$. Furthermore, by factoring out common factors in $\kk[t]$, we can assume that $A_0(\lambda),B_0(\lambda)\neq0$ for all $\lambda\in\kk$.
Since $f-\lambda$ and $g-\lambda$ are irreducible, it follows that $(f-\lambda)A(\lambda)=B(\lambda)(g-\lambda)$ is a coprime/comaximal relation. 

Now consider the system of affine equations over $\kk[t]$,
\begin{equation}\label{e:affsys}
(f-t)A(t)+B_0(t)B(t)=1,\qquad \deg A(t),\deg B(t)<\deg f
\end{equation}
in coefficients of $A(t),B(t)\in\kk[t]\otimes R$. 
By the preceding paragraph and Lemma \ref{l:comax_bds}(b), for every $\lambda\in\kk$, the specialization of the linear system \eqref{e:affsys} at $t=\lambda$ admits a solution. Furthermore, this solution is necessarily unique: if $a,b$ and $\tilde{a},\tilde{b}$ solve \eqref{e:affsys} at $t=\lambda$, then $(f-\lambda)(a-\tilde a)=B_0(\tilde b-b)$; irreducibility of $f-\lambda$ and degree comparison then imply $a-\tilde a=0$ and $\tilde b-b=0$. Let us transform the system \eqref{e:affsys} into a Smith normal matrix form,
$$\begin{pmatrix}
u_1&&& \\ &\ddots&& \\ && u_r&\phantom{a} \\ &&& \phantom{\vdots}
\end{pmatrix}z=
\begin{pmatrix}v_1\\ \vdots \\ v_r \\ \vdots
\end{pmatrix},\qquad u_j\in\kk[t]\setminus\{0\},\ v_j\in\kk[t],$$
where $z$ is the vector of coefficients of $A(t),B(t)$.
Since \eqref{e:affsys} at $t=\lambda$ is uniquely solvable, it follows that $v_j(\lambda)=0$ for $j>r$, and $u_j(\lambda)\neq0$ for $1\le j\le r$. Since this holds for every $\lambda\in\kk$, it follows that $u_1,\dots,u_r\in\kk\setminus\{0\}$, and $v_j=0$ for $j>r$. Thus, \eqref{e:affsys} admits a solution $A_1(t),B_1(t)\in\kk[t]\otimes R$. 

In the same way we see that there are $A_2(t),B_2(t)\in\kk[t]\otimes R$ satisfying $A_2(t)A_0(t)+B_2(t)(g-t)=1$. Together with \eqref{e:linsys} and \eqref{e:affsys}, this means that $(f-t)A_0(t)=B_0(t)(g-t)$ is a comaximal relation in $\kk[t]\otimes  R$. Hence, $f=g$ by \cite[Proposition 0.6.5]{cohn}.

Now, assume (b) holds. By \cite[Proposition 0.6.5]{cohn}, it suffices to see that $(f-t)a=a(g-t)$ is a comaximal relation in $\kk[t]\otimes R$.
Consider the system of affine equations over $\kk[t]$,
\begin{equation}\label{e:affsys1}
(f-t)A(t)+a B(t)=1,\qquad \deg A(t)<\deg a,\deg B(t)<\deg f
\end{equation}
in coefficients of $A(t),B(t)\in\kk[t]\otimes R$. 
By the assumption and Lemma \ref{l:comax_bds}(b), for every $\lambda\in\kk$, the specialization of the linear system \eqref{e:affsys1} at $t=\lambda$ admits a solution. Moreover, this solution is unique. Suppose, way of contradiction, that $c,d$ and $\tilde{c},\tilde{d}$ solve \eqref{e:affsys1} at $t=\lambda$, and $c\neq\tilde{c}$. Then $(f-\lambda)(c-\tilde c)=a(\tilde d-d)$, and so
$$(\tilde d-d)(c-\tilde c)^{-1} = a^{-1}(f-\lambda) = (g-\lambda)a^{-1}$$
in the universal skew field of fractions of $ R$.
Thus, there exists $0\neq q\in R$ such that $c-\tilde c=aq$ by left comaximality of $a$ and $g-\lambda$ \cite[Proposition 2.3.10]{cohn}, contradicting $\deg(c-\tilde c)<\deg a$.
From hereon, the same argument involving a Smith normal form as in the case (a) applies, showing that $(f-t)a=a(g-t)$ is comaximal in $\kk[t]\otimes R$, and thus $f=g$.
\end{proof}

The irreducibility assumption in Proposition \ref{p:irr_similar}(a) is essential, as the following example demonstrates.

\begin{exa}\label{exa:unexpected}
Denote
\begin{align*}
&a=yx^3y+xy+yx,\quad b=xyxyx+xy+yx,\\
&u=1+x^2y,\quad v=1+xyx,\quad w=1+yx^2.
\end{align*}
Observe that these five polynomials are irreducible, $bu=va$ and $av=wb$. Thus, $a,b$ are stably associated, and $u,v,w$ are stably associated. 
In particular, $ab,w$ are left coprime and $ba,u$ are right coprime. Hence, $(ab)u=w(ba)$ is a coprime relation.
By Theorem \ref{t:st_assoc} it follows that $ab$ and $ba$ are stably associated, yet $ab\neq ba$.\footnote{
To authors' best knowledge, the presented pair $a,b$ is the first known example of this phenomenon; cf. the open question in \cite[Exercise 0.5.4]{cohn}.
} 
Thus, $f=ab$ and $g=ba$ are distinct polynomials such that $f-\lambda$ and $g-\lambda$ are stably associated for every $\lambda\in\kk$ (though $f-t$ and $g-t$ are not stably associated in $\kk[t]\otimes\px$ by \cite[Proposition 0.6.5]{cohn}).
Since $f-\lambda$ and $g-\lambda$ are rank-equivalent for all $\lambda\in\kk$, it follows (by looking at the possible Jordan canonical forms) that $f$ and $g$ are pointwise similar on all pairs of $3\times 3$ matrices. On the other hand, if $\xi=\sqrt{29+13\sqrt{5}}$ and
$$X=\begin{pmatrix}
1&0&0&0\\0&\frac{\sqrt{5}-1}{2}&0&0\\0&0&-1&0\\
0&0&0&\frac{11-5\sqrt{5}}{4}\xi
\end{pmatrix},\quad
Y=\begin{pmatrix}
0&\frac{-5-\sqrt{5}}{10}&0&0\\
1&-1&2&0\\
1&\frac{-1}{2\sqrt{5}}&0&\frac{5-3\sqrt{5}}{10}\xi\\
\frac{\sqrt{5}-3}{4}&\frac{-\sqrt{5}}{2}&\frac{\sqrt{5}-3}{2}&\xi-4-2\sqrt{5}
\end{pmatrix},$$
then $f(X,Y)^2=0$ and $g(X,Y)^2\neq0$. 
Hence, $f(X,Y)$ and $g(X,Y)$ are not similar.\footnote{
There is no particular significance behind the numbers in $X$ and $Y$; this matrix pair was constructed essentially by brute force.
}
\end{exa}

On the other hand, Proposition \ref{p:irr_similar}(b) is one of the ingredients for showing that pointwise similar polynomials are equal. 
Another ingredient is the absence of nonzero homomorphisms between certain modules over $\px$, which is established by the following two lemmas.

\begin{lem} \label{lem:ext_t_neq_0}
Let $f \in R=\px$, and let $p, q \in \kk[t]$ be coprime univariate polynomials. Then, 
$$\Hom_R\big(R/Rp(f), R/Rq(f)\big) = \Ext^1_R\big(R/Rp(f), R/Rq(f)\big) = \{0\}.$$
In particular, $R/R(p(f)q(f)) \cong R/Rp(f) \oplus R/Rq(f)$.
\end{lem}

\begin{proof}
Let us consider the free resolution of the left module $R/Rp(f)$,
$$0\xrightarrow{\phantom{\ \cdot p(f)}}
R\xrightarrow{\ \cdot p(f)}R
\xrightarrow{\phantom{\ \cdot p(f)}}R/Rp(f)\xrightarrow{\phantom{\ \cdot p(f)}} 0.$$
Applying the functor $\Hom_R(\rule{2ex}{0.07ex}, R/Rq(f))$, we get the exact sequence
\begin{equation}\label{e:exact}
\begin{aligned}
0&\xrightarrow{\phantom{\ \cdot p(f)}}
&\Hom_R(R/Rp(f), R/Rq(f))
&\xrightarrow{\phantom{\ \cdot p(f)}}&R/q(f)R
&\xrightarrow{\ \cdot p(f)}
&R/q(f)R \\
&\xrightarrow{\phantom{\ \cdot p(f)}}
&\Ext^1_R(R/Rp(f), R/Rq(f))
&\xrightarrow{\phantom{\ \cdot p(f)}}&0\,.\qquad& &
\end{aligned}
\end{equation}
Here, we used the isomorphism $\Hom_R(R,R/Rq(f))\to R/Rq(f)$ determined by $\varphi\mapsto\varphi(1)$. 
Since $p$ and $q$ are coprime, there exist $a,b \in \kk[t]$, such that $pa + bq = 1$. Therefore, $a(f)p(f) + b(f)q(f) = 1$. So the induced map of right multiplication by $p(f)$ has an inverse given by right multiplication by $a(f)$. Therefore, the exact sequence \eqref{e:exact} implies the first part of the lemma. Since all extensions of $R/Rp(f)$ by $R/Rg(f)$ are trivial, the second part of the lemma follows.
\end{proof}

We also require a partial generalization of Lemma \ref{lem:ext_t_neq_0}.

\begin{lem}\label{l:tech_isospec_split}
If $f,g\in\px$ are isospectral and $p,q\in\kk[t]$ are coprime, then
$$\Hom_R\big(R/Rp(f), R/Rq(g)\big)=\{0\}.$$
\end{lem}

\begin{proof}
Suppose $\Hom_R(R/Rp(f), R/Rq(g))\neq\{0\}$. Then an irreducible factor $h$ of $p(f)$ is stably associated to a factor of $q(g)$ by \cite[Exercise 3.2.11]{cohn}. consequently to a factor of $q(f)$ by Lemma \ref{l:isospecfactors}. 
Thus, there exist distinct $\alpha,\beta\in\kk$ (the first a root of $p$, the second a root of $q$) such that a factor of $f-\alpha$ is stably associated to a factor of $f-\beta$. However, this is a contradiction by Lemma \ref{l:twoeig}. Hence, $\Hom_R(R/Rp(f), R/Rq(g))=\{0\}$.
\end{proof}

The next lemma allows us to reduce the pointwise similarity characterization to the case of non-composite polynomials.

\begin{lem}\label{lem:reduce_non_comp}
Every pair of pointwise similar polynomials in $\px$ is of the form $p(f),p(g)$ where $p\in\kk[t]$, and $f,g\in\px$ are non-composite and pointwise similar.
\end{lem}

\begin{proof}
By Lemma \ref{lem:isospec_non_comp}, pointwise similar polynomials are of the form $p(f),p(g)$ where $p\in\kk[t]$ and $f,g\in R$ are non-composite and isospectral. Now fix $\lambda \in \kk$ and note that there exists $k \in \N$, such that $p(t) - p(\lambda) = (t-\lambda)^k q(t)$, where $q\in\kk[t]$ and $(t-\lambda)^k$ are coprime. By the assumption and Lemma \ref{lem:ext_t_neq_0}, 
\begin{equation}\label{e:iso_split}
\begin{split}
R/R(f - \lambda)^k \oplus R/Rq(f) 
&\cong R/R(p(f)-p(\lambda)) \\
&\cong R/R(p(g) - p(\lambda)) \cong R/R(g - \lambda)^k \oplus R/Rq(g).
\end{split}
\end{equation}
Since $\Hom_R(R/R(f - \lambda)^k, R/Rq(g))=\{0\}$ and $\Hom_R(R/Rq(f),R/R(g - \lambda)^k)=\{0\}$ by Lemma \ref{l:tech_isospec_split}, the isomorphism \eqref{e:iso_split} implies $R/R(f-\lambda)^k \cong R/R(g-\lambda)^k$. 
Repeating the above argument with powers of $p - p(\lambda)$, we see that 
$R/R(f-\lambda)^{\ell k} \cong R/R(g-\lambda)^{\ell k}$ for all $\ell \in \N$. By Lemma \ref{l:power}, $(f-\lambda)^{\ell}$ is stably associated to $(g-\lambda)^{\ell}$. 
Thus, $\rk (f(\uX)-\lambda I)^{\ell}=\rk (g(\uX)-\lambda I)^{\ell}$ for all $\lambda\in\kk$, $\ell\in\N$ and matrix tuples $\uX$. Using the Jordan canonical form we then conclude that $f$ and $g$ are pointwise similar.
\end{proof}

We are now ready to show that pointwise similarity is equality.

\begin{thm}\label{t:similar}
Let $f, g \in \px$ be pointwise similar. Then, $f = g$.
\end{thm}

\begin{proof}
By Lemma \ref{lem:reduce_non_comp}, we may assume that $f$ and $g$ are non-composite. By \cite[Theorem 3.2]{vol20} there exists $\beta\in\kk\setminus\{0\}$ such that $f-\beta$ and $g-\beta$ are irreducible. 
Let $a \in R\setminus\{0\}$ denote an intertwinner of $f$ and $g$ of minimal degree; note that $(f-\beta)a=a(g-\beta)$ is a coprime relation. Since $f$ and $g$ are pointwise similar, $f(f-\beta)$ is stably associated to $g(g-\beta)$. Then,
\begin{equation}\label{e:quad}
R/Rf \oplus R/R(f-\beta) \cong R/R(f(f-\beta)) \cong R/R(g(g-\beta)) \cong R/Rg \oplus R/R(g-\beta)
\end{equation}
by Lemma \ref{lem:ext_t_neq_0}. The middle isomorphism in \eqref{e:quad} is given by $1 \mapsto a_\beta$ with $f(f-\beta) a_\beta = b_\beta g(g-\beta)$ comaximal. Stable association of $f-\beta$ and $g-\beta$ implies 
$\Hom_R(R/Rf,R/R(g-\beta))\cong \Hom_R(R/Rf,R/R(f-\beta))=\{0\}$
by Lemma \ref{lem:ext_t_neq_0}; similarly, stable association of $f$ and $g$ implies $\Hom_R(R/Rg,R/R(f-\beta))=\{0\}$.
Thus, the maps $R/Rf \to R/Rg$ and $R/R(f-\beta) \to R/R(g-\beta)$ induced by \eqref{e:quad} are isomorphisms. Since \eqref{e:quad} maps $(1,1)$ to $(a_\beta,a_\beta)$, there exist $c_\beta, d_\beta \in R$ such that 
\[
f a_\beta = c_\beta g \quad\text{and}\quad 
(f-\beta) a_\beta = d_\beta (g-\beta)
\]
are comaximal relations. 
Combining these relations gives $\beta (a_\beta - d_\beta) = (c_\beta - a_\beta) g$. 
Since $f-\beta$ and $g-\beta$ are irreducible, Theorem \ref{t:eigering} implies $\dim_{\kk} \Hom_R(R/R(f-\beta), R/R(g-\beta)) = 1$. Therefore, up to rescaling, we may assume that $a_\beta = a + h (g-\beta)$ and $d_\beta = a + (f-\beta) h$ for some $h\in R$. 
Thus, $a_\beta - d_\beta = h(g - \beta) - (f - \beta) h = h g - f h$, and consequently,
\[
(c_\beta - a_\beta) g = \beta (a_\beta - d_\beta) 
= \beta h g - \beta f h.
\]
Hence, $\beta f h = (a_\beta - c_\beta + \beta h) g$. On the other hand,
\[
c_\beta g = f a_\beta = f (a + h (g -\beta)) = f a + f h g - \beta f h = a g + f h g - \beta f h.
\]
Hence, $\beta f h = (a - c_\beta + f h) g$. Combining the two, we get $(a_\beta - c_\beta + \beta h) g = (a - c_\beta + f h) g$. Therefore, $a_\beta = a + (f-\beta) h = d_\beta$, and consequently $a_\beta = c_\beta$, so $f$ and $g$ are comaximally intertwined (i.e., $fa_\beta=a_\beta g$ is comaximal). 

Applying the above argument to $f-\lambda$ and $g-\lambda$ (in place of $f$ and $g$) for arbitrary $\lambda \in \kk$, we conclude that there exists $w_\lambda\in R$ such that
$(f - \lambda)w_\lambda=w_\lambda (g-\lambda)$ is a comaximal relation. Moreover, we may assume that $w_\lambda$ is not left-divisible by an element of $R$ that commutes with $f$ (otherwise, dividing $w_\lambda$ on the left with it again gives $w_\lambda'$ such that $(f - \lambda)w_\lambda'=w_\lambda' (g-\lambda)$ is comaximal).
Since $fw_\lambda=w_\lambda g$, Remark \ref{r:unique} implies that $w_\lambda$ is a scalar multiple of $a$. Hence, $(f - \lambda)a=a(g-\lambda)$ is comaximal for every $\lambda$, so $f=g$ by Proposition \ref{p:irr_similar}(b).
\end{proof}

As a consequence, we obtain a criterion for witnessing noncommutativity in the free algebra through matricial evaluations.

\begin{cor}
If $a,b\in\px$ do not commute, there exists a matrix tuple $\uX$ such that the ranks of $\big(a(\uX)b(\uX)\big)^k$ and $\big(b(\uX)a(\uX)\big)^k$ differ for some $k\in\N$.
\end{cor}

\begin{proof}
Note that $(ab-\lambda)^k$ and $(ba-\lambda)^k$ are stably associated for every $k\in\N$ and $\lambda\in\kk\setminus\{0\}$, due to the comaximal relation $(ab-\lambda)^k a=a(ba-\lambda)^k$. If $\rk\big(a(\uX)b(\uX)\big)^k=\rk \big(b(\uX)a(\uX)\big)^k$ for all $\uX$ and $k\in\N$, then $\rk ((ab)(\uX)-\lambda I)^{k}=\rk ((ba)(\uX)-\lambda I)^{k}$ for all $\lambda\in\kk$, $k\in\N$ and matrix tuples $\uX$. Hence, the Jordan canonical form shows that $ab$ and $ba$ are pointwise similar, so $ab=ba$ by Theorem \ref{t:similar}.
\end{proof}

\section{Norm and singular value equivalence}\label{sec:real}

Throughout this section, let $\kk=\C$. The final main result of this paper (Theorem \ref{t:realmix}) characterizes pairs of polynomials that have pointwise equal norms. Investigating this norm equivalence leads one to look at a suitable real structure on the free algebra. 
In addition to $\cx$, we consider its extension, the free $*$-algebra $\cxs$ generated by $x_1,\dots,x_n,x_1^*,\dots,x_n^*$, and endowed the antilinear involution $*$ determined by $x_j\leftrightarrow x_j^*$. Note that $f^*(\uX,\uX^*)=\left(f(\uX,\uX^*)\right)^*$ for all $\uX\in\mtxc{k}^n$.

The following is a well-known fact in complex analysis, which we record for convenience.

\begin{lem}\label{l:classic}
Let $\cO\subseteq\C^n$ be an Euclidean open set, and let $\cD$ be an analytically dense subset of $\cO$. If $\phi,\psi$ are nonzero analytic functions on $\cO$ such that $|\phi(\underline{z})|=|\psi(\underline{z})|$, then there exists $\zeta\in\C$ with $|\zeta|=1$ such that $\psi=\zeta \phi$.
\end{lem}

\begin{proof}
Let $V=\{\phi=0\}\subset\cO$. Then $\frac{\psi}{\phi}$ is analytic on the open set $\cO\setminus V$, and $|\frac{\psi}{\phi}|=1$ on $\cD\cap(\cO\setminus V)$. Since $\cD\cap(\cO\setminus V)$ is analytically dense in $\cO\setminus V$, it follows that $|\frac{\psi}{\phi}|=1$ on $\cO\setminus V$. Since the latter is open, the maximum modulus principle implies that $\frac{\psi}{\phi}$ is constant, so $\psi=\zeta\phi$ for some $\zeta\in\C$ with $|\zeta|=1$.
\end{proof}

The main message of the following proposition is that noncommutative polynomials with pointwise equal spectral radii are isospectral up to scaling by a constant of modulus 1. For a later use, we establish a slightly stronger statement, where $\her{k}$ denotes the real space of hermitian $k\times k$ matrices.

\begin{prop}\label{p:radius}
Let $f,g\in\cx$. Assume that for every $k\in\N$ and $\uX\in\her{k}^n$, $f(\uX)$ and $g(\uX)$ have the same spectral radius. Then there exists $\zeta\in\C$ with $|\zeta|=1$ such that $\zeta f$ and $g$ are isospectral.
\end{prop}

\begin{proof}
The statement is clear if either $f$ or $g$ is constant; thus, we assume that $f,g\in\cx\setminus\C$.
Let $\cO_k$ denote the set of $\uX\in\mtxc{k}^n$ such that $f(\uX)$ and $g(\uX)$ each have $n$ pairwise distinct eigenvalues. By \cite[Corollary 2.10]{BreVol}, there is a cofinite $K\subseteq\N$ such that for all $k\in K$, the set $\cO_k$ is nonempty and Zariski open in $\mtxc{k}^n$. 
For $k\in K$ and $\uX\in\cO_k$, the eigenvalues of $f(\uX),g(\uX)$ can be viewed as algebraic functions $\cO_k\to\C$. 
Then there exists a Euclidean open subset $\cO_k'\subset\cO_k$ and holomorphic functions $\lambda_k,\mu_k$ on $\cO'_k$ such that $\lambda_k(\uX)$ (resp. $\mu_k(\uX)$) is a maximal (by absolute value) eigenvalue of $f(\uX)$ (resp. $g(\uX)$), for all $\uX\in\cO'_k$.
By assumption, $|\lambda_k|=|\mu_k|$ on $\cO'_k\cap\her{k}^n$, which is analytically dense in $\cO'_k$.
By Lemma \ref{l:classic}, there exists $\zeta_k\in\C$ with $|\zeta_k|=1$ such that $\mu_k=\zeta_k\lambda_k$. That is, $\zeta_k f(\uX)$ and $g(\uX)$ share an eigenvalue for all $\uX$ in a Zariski dense subset of $\mtxc{k}^n$.
As in the proof of Theorem \ref{t:isospec}(i)$\Rightarrow$(ii), we see that
the spectra of $\zeta_k f(\uX)$ and $g(\uX)$ coincide for every $\uX\in\mtxc{k}^n$.
Now fix some $m\in K$ and $\uX\in\mtxc{m}^n$ with $f(\uX)\neq\{0\}$ (e.g., $\uX\in\cO'_m$). 
Then, the spectrum of $f(\uX)$ satisfies
$$\zeta_m\cdot\sigma\big(f(\uX)\big)
=\sigma\big(g(\uX)\big)=\sigma\big(g(\uX^{\oplus\ell})\big)
=\zeta_{\ell m}\cdot\sigma\big(f(\uX^{\oplus\ell})\big)
=\zeta_{\ell m}\cdot\sigma\big(f(\uX)\big)
$$
for all $\ell\in\N$. Thus, there is $\zeta\in\C$ with $|\zeta|=1$ such that $\zeta=\zeta_k$ for infinitely many $k\in K$. Hence, $\zeta f$ and $g$ are isospectral.
\end{proof}

The following theorem combines the results on isospectrality from Section \ref{sec:eig} with 
Zariski densensess of hermitian tuples.

\begin{thm}\label{t:realmix}
The following are equivalent for $f,g\in\cx$:
\begin{enumerate}[(i)]
\item $f$ and $g$ have pointwise the same operator norm;
\item $f$ and $g$ have pointwise the same singular values;
\item $g=\zeta f$ for some $\zeta\in\C$ with $|\zeta|=1$.
\end{enumerate}
\end{thm}

\begin{proof}
The implications (iii)$\Rightarrow$(ii)$\Rightarrow$(i) are clear. 
Let $\uy=(y_1,\dots,y_n)$ and $\uz=(z_1,\dots,z_n)$ be two tuples of freely noncommuting self-adjoint variables. Viewing $f,g\in\cx$ as elements in $\cxs$, we consider $F=f(\uy+i\uz)f^*(\uy+i\uz)$ and $G=g(\uy+i\uz)g^*(\uy+i\uz)$ in $\C\!\Langle \uy, \uz \Rangle$.

(i)$\Rightarrow$(ii): 
Since $f$ and $g$ have the same operator norm on $\mtxc{k}^n$, $F$ and $G$ have the same spectral radius on $\her{k}^{2n}$.
By Proposition \ref{p:radius}, there exists $\xi\in\C$ with $|\xi|=1$ such that $\zeta F$ and $G$ are isospectral. Furthermore, $\xi=1$ since $F$ and $G$ are positive semidefinite on $\her{k}^{2n}$. Thus, $f$ and $g$ have the same singular values on $\mtxc{k}^n$, for every $k\in\N$.

(ii)$\Rightarrow$(iii):  
By assumption, $F$ and $G$ have the same eigenvalues on $\her{k}^{2n}$. Since $\her{k}^{2n}$ is Zariski dense in $\mtxc{k}^{2n}$ and eigenvalues are algebraic functions, it follows that $F$ and $G$ have the same eigenvalues on $\mtxc{k}^{2n}$, for every $k$.
By Theorem \ref{t:isospec}, there is a nonzero $A\in \C\!\Langle \uy, \uz \Rangle$ such that $FA=AG$. So, $ff^*a=agg^*$ for some nonzero $a\in \cxs$; we may assume that $\deg a$ is minimal. Since the variables $\ux$ and $\ux^*$ are free, it is easy to see that either $a$ is constant, or $ff^*$ is its left divisor. By minimality it is constant, and so $ff^*=gg^*$. Hence, $g=\zeta f$ for some $\zeta\in\C$ with $\zeta \bar{\zeta}=1$.
\end{proof}

Looking at noncommutative polynomials in free variables and their formal adjoints naturally leads to further pointwise equivalences. For example, while pointwise (unitarily) similar polynomials in $\cx$ are necessarily equal by Theorem \ref{t:similar}, there are non-trivial pairs of pointwise unitarily similar polynomials in $\cxs$, such as $xx^*$ and $x^*x$. Such pointwise equivalences on $\cxs$ have a distinct real flavor, and admit larger equivalence classes than their $*$-free analogs presented in earlier sections (for example, various nonconstant polynomials in $\cxs$ are rank-equivalent to 1, cf. \cite{KPV}). Their exploration calls for methods beyond the scope of this paper.

\bibliographystyle{abbrv}
\bibliography{ncpoly}

\end{document}